\documentclass[12pt]{amsart}

\usepackage[OT1]{fontenc}
\usepackage{amsmath, mathrsfs, amsfonts, amssymb, enumerate}
\usepackage[numbers]{natbib}
\usepackage[margin=0.75in]{geometry}
\usepackage[singlelinecheck=false]{caption}

\usepackage{nameref}
\usepackage{zref-xr,zref-user}
\zxrsetup{toltxlabel}

\usepackage{tikz}
\usetikzlibrary{arrows, calc}
\usetikzlibrary{patterns, shapes, decorations.markings}

\usepackage[colorlinks,citecolor=blue,urlcolor=blue]{hyperref}
\usepackage{pgfplots}
\usetikzlibrary{external}
\numberwithin{equation}{section}

\numberwithin{equation}{section}
\theoremstyle{definition}
\newtheorem{definition}[equation]{Definition}
\theoremstyle{definition}
\newtheorem{remark}[equation]{Remark}
\theoremstyle{definition}
\newtheorem{convention}[equation]{Convention}
\theoremstyle{definition}

\theoremstyle{lemma}

\theoremstyle{lemma}
\newtheorem{lemma}[equation]{Lemma}
\theoremstyle{theorem}
\newtheorem{theorem}[equation]{Theorem}
\theoremstyle{proposition}
\newtheorem{proposition}[equation]{Proposition}
\theoremstyle{corollary}
\newtheorem{corollary}[equation]{Corollary}

\theoremstyle{corollary}
\theoremstyle{definition}

\theoremstyle{example}

\theoremstyle{proposition}

\theoremstyle{definition}

\newcommand{\E}{\mathbf{E}}
\newcommand{\PP}{\mathbf{P}}
\newcommand{\R}{\mathbf{R}}
\newcommand{\Q}{\mathsf{Q}}
\renewcommand{\P}{\mathsf{P}}
\newcommand{\C}{\mathbf{C}}
\newcommand{\N}{\mathbf{N}}

\renewcommand{\d}{d}

\renewcommand{\t}{{\tau^*}}
\zexternaldocument*[I-]{polyZ}

\title[Noise-Induced Stabilization of Planar Flows II]{Noise-Induced
  Stabilization of Planar Flows II}

\author[D. P. Herzog]{David P. Herzog}
\author[J. C. Mattingly]{Jonathan C. Mattingly}
\makeindex
\begin{document}
\maketitle

\begin{abstract}
  We continue the work started in Part I \cite{HerzogMattingly2013I}, showing how the addition of noise can stabilize an otherwise unstable system.  The
  analysis makes use of nearly optimal Lyapunov functions. In this
  continuation, we remove the main limiting assumption of Part I by an
  inductive procedure as well as establish a lower bound which shows that
  our construction is radially sharp. We also prove a version of Peskir's \cite{Peskir_07}
generalized Tanaka formula adapted to  patching together Lyapunov
functions. This greatly simplifies the analysis used in previous works. 
 \end{abstract} 

\section{Introduction}

In Part I of this work \cite{HerzogMattingly2013I}, we investigated the following complex-valued dynamics
\begin{equation}\label{eqn:genpoly}
\left\{
  \begin{aligned}
    dz_t &=(a z_t^{n+1}+ a_n z_t^n + \cdots + a_0) \, dt + \sigma \, dB_t\\
    z_0 &\in \C
  \end{aligned}\right.
\end{equation}
where $n\geq 1$ is an integer, $a\in \C\setminus \{0\}$, $a_i \in \C$, $\sigma \geq 0$, and
$B_t=B_t^{(1)} + i B_t^{(2)}$ is a complex Brownian motion defined on
a probability space $(\Omega, \mathcal{F}, \PP)$.  There, we studied
how the presence of noise ($\sigma >0$ in \eqref{eqn:genpoly}) could
stabilize the unstable underlying deterministic system ($\sigma=0$ in
\eqref{eqn:genpoly}).
To prove stability in the stochastic perturbation, we developed a
framework for building Lyapunov functions and applied it to
\eqref{eqn:genpoly} assuming that the drift in equation
\eqref{eqn:genpoly} did not contain any ``significant" lower-order
terms; that is, we assumed that $a_j=0$ for $\lfloor \frac{n}{2} \rfloor \leq j \leq
n$.  This was done in order to focus on the overarching elements of
the construction of Lyapunov functions and to avoid any additional
complexities caused by the presence of such lower-order terms.  In
this paper, we give an inductive asymptotic argument which shows how
to remove this assumption, thereby proving the full version of
Theorem~\ref{I-thm:decayStMeasure} of Part I
\cite{HerzogMattingly2013I}.  Here, we also provide a radially sharp
lower bound on the decay rate of the invariant measure's density as
stated in Theorem~\ref{I-thm:limit_bound} of Part I
\cite{HerzogMattingly2013I}. This work extends and strengthens a
stream of results on similar problems \cite{AKM_12, BodDoe_12, GHW_11,
  Her_11, Sch_93}.

As first glance, it is surprising that the general case is
substantially more complicated than those cases covered in Part I
\cite{HerzogMattingly2013I}, as intuition suggests that the behavior
of the process $z_t$ at infinity is determined by the leading-order
term $z^{n+1}$ and the noise.  We will see here, however, that there
is a range in which each of the intermediate lower-order terms becomes
dominant in the angular direction at infinity as one moves towards to regions where noise
dominates.  The scaling analysis of Section~\ref{I-sec:PiecewiseI} of
Part I \cite{HerzogMattingly2013I} hinted at this possibility when we
employed our simplifying assumption, for it implied that the dominant
balance of terms transferred directly from the leading order term
$z^{n+1}$ to the angular diffusion term without any interference from
the remaining lower-order terms. In this paper, we will perform the
analogous analysis for the general case in
Section~\ref{sec:PiecewiseBeh}, showing how to correctly study the
process at infinity in the presence of the intermediate lower-order
terms.  We will see, in particular, that the analysis used in
Section~\ref{I-sec:PiecewiseI} of Part I \cite{HerzogMattingly2013I}
breaks down in ``small" regions containing the explosive trajectories
of the deterministic system ($\sigma =0$ in equation
\eqref{eqn:genpoly} ) and that the additional terms produce
intermediate boundary layers which surround the inner most layer where
noise dominates.

We begin in Section~\ref{sec:preliminaries} by recalling the general
setup of Part I \cite{HerzogMattingly2013I}.  There, we also state the
main results we will prove in this paper. In
Section~\ref{sec:PiecewiseBeh}, we perform the asymptotic analysis
which guides and motivates the rest of the work.  Specifically, we
will use the asymptotically dominant operators yielded from it to
define our Lyapunov functions in Section~\ref{sec:constructionPsigen}
by using a succession of associated PDEs based on these operators. In
Section~\ref{sec:boundaryFluxCalculations}, we analyze boundary flux
terms in order to show that the local Lyapunov functions can be
patched together to produce a global Lyapunov function. Using these
calculations, we verify the needed global Lyapunov structure in
Section~\ref{sec:CheckingTheGlobalLyapunovBounds}. In
Section~\ref{sec:optimality}, we show that the family of Lyapunov
functions we have constructed are radially optimal by establishing a
matching lower bound at infinity of the invariant probability density
function. In Section~\ref{sec:peskir}, we prove a version of Peskir's
generalized Tanaka formula \cite{Peskir_07} which allows to avoid
$C^2$-smoothing along the boundaries of the local Lyapunov functions.
Being able to avoid such smoothing greatly simplifies former similar
analyses \cite{AKM_12, GHW_11, Her_11}. In
Section~\ref{sec:concluions}, we make some concluding remarks and
suggestions for possible directions of future research.

\section{Preliminaries}
\label{sec:preliminaries}
  In this section, we will both recall the general setup of Part I~\cite{HerzogMattingly2013I} and state the main results to be proved in this paper.  Throughout this remainder of this work, we will study more generally the complex-valued SDE
\begin{align} 
  \label{eqn:polyZ} 
  dz_t = [a z_t^{n+1} + F(z_t,\bar{z}_t)]\, dt + \sigma \, dB_t
\end{align} 
where $a\in \C\setminus \{ 0\}$, $n\geq 1$, $\sigma>0$, $B_t=B_t^1 + i B_t^2$ is a complex
Brownian motion and $F(z, \bar{z})$ is a complex polynomial in the
variables $(z, \bar{z})$ with $F(z,\bar{z}) =\mathcal{O}(|z|^n)$
as $|z|\rightarrow \infty$.  This is a slight generalization of the system \eqref{eqn:genpoly} in that $F(z, \bar{z})$ need not be a complex polynomial in the variable $z$ only.     

The main goal of this work is to prove the following result.
\begin{theorem}\label{thm:decayStMeasure} The Markov process defined
  by \eqref{eqn:polyZ} is non-explosive and possesses a unique
  stationary measure $\mu$.  In addition, $\mu$ satisfies:
\begin{align*}
\int_\C (1+|z|)^\gamma \, d\mu(z) < \infty \,\,\text{ if and only if }\,\, \gamma < 2n.  
\end{align*}
Furthermore, $\mu$ is
ergodic and has a probability density function $\rho$ with respect to
Lebesgue measure on $\R^2$ which is smooth and everywhere positive.
 \end{theorem}

In addition to proving Theorem~\ref{thm:decayStMeasure}, we will also characterize the convergence of the process $z_t$ defined by \eqref{eqn:polyZ} to the unique stationary measure $\mu$.  To state this result, for any measurable function $w\colon \C
 \rightarrow [1,\infty)$, let $\mathcal{M}_w(\C)$ denote the set of probability measures $\nu$ on $\C$ satisfying $w\in L^1(\nu)$  
  and define the weighted total variation metric
 $d_w$ on $\mathcal{M}_w(\C)$ by
\begin{align*}
  \d_w(\nu_1,\nu_2)=\sup_{\substack{\phi:\C \rightarrow \R
      \\|\phi(z)|\leq w(z)} }\bigg[\int \phi(z)\, \nu_1(dz)- \int
  \phi(z) \, \nu_2(dz)\bigg].  
\end{align*}

\begin{theorem}\label{thm:convergence} Let $P_t$ denote the Markov
  semi-group corresponding to \eqref{eqn:polyZ} and let $\alpha \in (0, n)$ be arbitrary.   Then there exists a function $\Psi \colon \C
  \rightarrow [0,\infty)$ and positive constants $c, d, K$ such that 
\begin{align*}
  c |z|^\alpha \leq \Psi(z) \leq d |z|^{\alpha + \frac{n}{2}+1}  
\end{align*}
for all $|z| \geq K$ and such that if $w(z)=1+\beta \Psi(z)$ for some $\beta >0$, then ${\nu P_t \in \mathcal{M}_w(\C)}$ for all $t>0$ and any probability measure $\nu$ on $\C$.  Moreover, with the same choice of $w$, there exist positive constants $C, \gamma$ such that for any two probability measures $\nu_1, \nu_2$ on $\C$ and any $t\geq 1$
  \begin{align*}
   d_w(\nu_1 P_t,\nu_2 P_t) \leq C e^{-\gamma t} \| \nu_1-\nu_2\|_{TV}.  
 \end{align*} 
\end{theorem}

Most of the results stated above will be established by constructing certain types of Lyapunov functions.  In Part I~\cite{HerzogMattingly2013I} of this work, however, we used a slightly more general formulation of a Lyapunov function than usually employed in existing literature.  Therefore, we now recall what we mean by \emph{Lyapunov pairs} as introduced in Section~\ref{I-sec:conseq_lyap} of Part I~\cite{HerzogMattingly2013I}.

\begin{definition}
\label{def:Lyapunov}
Let $\xi_t$ denote a time-homogeneous It\^{o} diffusion on $\R^k$ with $C^\infty$ coefficients and define stopping times $\tau_n= \inf\{t>0 \,: \, |\xi_t| \geq n\}$ for $n\in \N$.        
Let $\Psi, \Phi: \R^k\rightarrow [0, \infty)$ be continuous. Then we call $(\Psi, \Phi)$
a \textsc{Lyapunov pair corresponding to} $\xi_t$ if:
  \begin{enumerate}
  \item $\Psi(\xi)\wedge \Phi(\xi) \rightarrow \infty$ as
    $|\xi|\rightarrow \infty$;
  \item\label{lyac:b}  There exists a locally bounded and measurable function ${g: \R^k \rightarrow \R}$ such that the following equality holds for all $\xi_0 \in \R^k$, $n\in \N$ and all bounded stopping times $\upsilon$:
  \begin{align*}
   \E_{\xi_0} \Psi(\xi_{\upsilon \wedge \tau_n}) = \Psi(\xi_0) +
   \E_{\xi_0}\int_0^{\upsilon \wedge \tau_n} g(\xi_s) \, ds + \text{Flux}(\xi_0, \upsilon, n)  
 \end{align*}
 where $\text{Flux}(\xi_0, \upsilon, n)\in (-\infty, 0]$ and $\text{Flux}(\xi_0, t, l) \leq \text{Flux}(\xi_0, s, n)$ for all $0\leq s\leq t$, $n\leq l$, $\xi_0 \in \R^k$.    
\item \label{lyac:c}There exist constants $m, b>0$ such that for all $\xi \in \R^k$  
  \begin{align*}
  g(\xi) \leq - m \Phi(\xi) + b.    
  \end{align*} 
    \end{enumerate}
  The function $\Psi$ in a Lyapunov pair $(\Psi, \Phi)$ is called a
  \textsc{Lyapunov function.  }
  \end{definition}
For an explanation of the differences between the usual notion of a Lyapunov function and the notion used here, consult Remark~\ref{I-rem:C2} of Part I~\cite{HerzogMattingly2013I}.

Most of the paper will be spent proving the following result giving the existence of certain types of Lyapunov pairs corresponding to the dynamics \eqref{eqn:polyZ}.

\begin{theorem} \label{thm:lyapfunmadness} For each $\gamma \in (n,
2n)$ and $\delta=\delta_\gamma >0$ sufficiently small, there exist
Lyapunov pairs $(\Psi, \Psi^{1+\delta})$ and $(\Psi, |z|^\gamma)$
corresponding to the dynamics \eqref{eqn:polyZ} such that the bound 
\begin{align*} c|z|^{\gamma-n}\leq \Psi(z) \leq d|z|^{\gamma-n
+\frac{n}{2}+1} \end{align*} is satisfied for all $|z| \geq K$ for some positive constants
$c,d, K$.  \end{theorem}

By the results of Section~\ref{I-sec:conseq_lyap} of Part I
\cite{HerzogMattingly2013I}, Theorem \ref{thm:lyapfunmadness} implies
almost all of the main results.  In particular, all consequences of
Theorem~\ref{thm:decayStMeasure} and Theorem~\ref{thm:convergence} follow except \begin{align} \label{eqn:lowerboundinvm}
  \int_\C (1+|z|)^\gamma \, d\mu(z) = \infty \, \, \text{ if } \,\,
  \gamma \geq 2n. \end{align}
  To prove \eqref{eqn:lowerboundinvm}, we will show the following stronger result.

\begin{theorem} \label{thm:limit_bound} Let $\rho(x,y)$ denote the
invariant probability density function of \eqref{eqn:polyZ} with
respect to Lebesgue measure on $\R^2$.  Then
there exist positive constants $c, K$ such
that \begin{align} \label{eqn:limit_bound} |(x,y)|^{2n+2} \rho(x, y)
\geq c\,\, \text{ for } \, \, |(x,y)|\geq K  \end{align} where $|(x, y)|=\sqrt{x^2 +y^2}$ denotes the standard Euclidean distance on $\R^2$. \end{theorem}

Throughout, we will assume that the reader
is familiar with Section~\ref{I-sec:GenOver} of Part I \cite{HerzogMattingly2013I} which gives the general outline
of the construction procedure used to produce Lyapunov pairs.  These Lyapunov pairs will be constructed using this procedure in Sections~\ref{sec:PiecewiseBeh}-\ref{sec:CheckingTheGlobalLyapunovBounds}, thus proving Theorem~\ref{thm:lyapfunmadness}.  In Section \ref{sec:optimality}, we change our focus from constructing Lyapunov pairs to proving Theorem
\ref{thm:limit_bound}.  Section \ref{sec:gen_ito} contains the proof
of a version of Peskir's result \cite{Peskir_07}.

\begin{remark}
\label{rem:ae1}
Throughout the proofs of the main results, we will assume without loss of generality that $a=1$ in equation \eqref{eqn:polyZ}.  Indeed one can get from either system to the other by multiplying the solution by a non-zero complex constant and using the fact that $e^{i\theta}B_t$, $\theta \in \R$, is also a complex Brownian motion.  
\end{remark}

\section{The Asymptotic operators and their associated regions
} \label{sec:PiecewiseBeh}

As in Part I \cite{HerzogMattingly2013I}, we will identify the
asymptotically dominant terms in equation \eqref{eqn:polyZ} at
infinity by analyzing the time-changed Markov generator $L$ of the process $z_t$ as
$r=|z|\rightarrow \infty$.  Because doing this is substantially more involved than in Part I~\cite{HerzogMattingly2013I}, we have provided a summary of the analysis that follows in Section~\ref{subsec:ingeneral} and added Figure~\ref{fig1} to help illustrate the regions and the corresponding deterministic flow. 

By Remark~\ref{rem:ae1}, we may assume without loss of generality that $a=1$ in equation \eqref{eqn:polyZ} throughout the analysis.  Hence, after making the time change $t \mapsto \tau
= \int_0^t |z_s|^n ds$, the time-changed generator $L$ has
the following form when written in polar coordinates $(r, \theta)$:
\begin{align} 
  \label{eqn:genL} 
  L = r \cos(n\theta)\partial_r + \sin(n\theta) \partial_{\theta} +
  P(r, \theta)\partial_r + Q(r, \theta) \partial_{\theta} +
  \frac{\sigma^2}{2 r^n }\partial_{r}^2 +
  \frac{\sigma^2}{2r^{n+2}}\partial_{\theta}^2
\end{align}
where 
   \begin{align}\label{eqn:PQ} 
    P(r, \theta) = \sum_{k=0}^{n+2} r^{k-n-2} f_k(\theta)  \,\,\,\, \text{ and } \,\,\,\, 
    Q(r, \theta)  = \sum_{k=0}^{n+1} r^{k-n-2} g_k(\theta) 
  \end{align}
for some collection of smooth real-valued functions $f_k$, $g_k$ which are $2\pi$-periodic.   
As we recall, the inclusion of the $k=0$ terms is not needed to encapsulate all terms in the generator of the process \eqref{eqn:polyZ}.  However, their presence is required to deal with a secondary calculation needed in the proof of Theorem~\ref{thm:limit_bound}.

As in Section~\ref{I-subsec:Bas_structI} and
Section~\ref{I-sec:reductions} of \cite{HerzogMattingly2013I}, we
will build our Lyapunov function for all $(r,\theta)$ by restricting analysis of $L$ to the principal wedge 
\begin{equation*}
  \mathcal{R}=\{(r, \theta) \, : \, r\geq r^*, \, -\tfrac{\pi}{n}\leq \theta \leq \tfrac{\pi}{n} \}.  
\end{equation*}
In \cite{HerzogMattingly2013I}, we recall that to do this construction, we divided $\mathcal{R}$ into four regions: a ``priming'' region $\mathcal{S}_0$ to initialize
the construction, a transport region $\mathcal{S}_1$, a
transition region $\mathcal{S}_2$ to blend between the transport region $\mathcal{S}_1$ and a region $\mathcal{S}_3$ where the noise still plays a role at infinity.  Moreover, this division of the principal wedge $\mathcal{R}$ was implied by the asymptotic analysis of $L$ carried out in Section \ref{I-sec:PiecewiseI} of \cite{HerzogMattingly2013I}.  Here, too, we will see that a division of $\mathcal{R}$ holds in the general case, but this time there are many more regions.  Initially, the analysis in the general case will exactly coincide with the analysis done previously.  Specifically, the first two regions, $\mathcal{S}_0$ and $\mathcal{S}_1$, will be of the same form as before.  Afterwards, however, the dynamics at infinity undergoes further incremental changes, and this results in a significant increase in the number of regions and, consequently, the number of asymptotic operators.

As in  Section~\ref{I-sec:PiecewiseI} of \cite{HerzogMattingly2013I},  
we introduce the family of scaling transformations 
\begin{align*} 
  S_\alpha^\lambda : (r, \theta) \mapsto (\lambda r, \lambda^\alpha 
  \theta) 
\end{align*} 
where $\lambda >0$ and $\alpha \geq 0$.  This is done to facilitate the identification of the dominant balances in $L$ as $r\rightarrow \infty$ with $(r, \theta)
\in \mathcal{R}$.  Operationally, we study the scaling properties  of
\begin{multline*}
L\circ S_\alpha^\lambda(r, \theta)= r \cos(n\theta  \lambda^{-\alpha}
) \partial_r  +  \lambda^{\alpha}  \sin(n\theta  \lambda^{-\alpha}
) \partial_{\theta} + \lambda^{-2-n} \frac{\sigma^2}{2 r^n
}\partial_{r}^2 \\+
\lambda^{2\alpha-n-2}\frac{\sigma^2}{2r^{n+2}}\partial_{\theta}^2  +
\lambda^{-1} P(\lambda r, \lambda^{-\alpha} \theta)\partial_r +
\lambda^\alpha Q(\lambda r, \lambda^{-\alpha}\theta) \partial_{\theta}.   
\end{multline*} 
as $\lambda\rightarrow \infty$ for different choices of $\alpha \geq  0$.

As done in Section \ref{I-sec:PiecewiseI} of Part I \cite{HerzogMattingly2013I}, we begin by studying $L$ as
$r\rightarrow \infty$, $(r, \theta) \in \mathcal{R}$, in regions where
$|\theta|$ is bounded away from zero; that is, we first consider
$L\circ S_0^\lambda$ as $\lambda \rightarrow \infty$. Observing that 
\begin{align*} 
  L\circ S_0^\lambda = r\cos(n\theta) \partial_r +
  \sin(n\theta) \partial_\theta + O(\lambda^{-1}) \text{ as } \lambda
  \rightarrow \infty,
\end{align*}
we still expect 
\begin{align} 
  T_1 = r\cos(n\theta) \partial_r + \sin(n\theta) \partial_\theta
\end{align}
to satisfy $L \approx T_1$ for $r\gg 0$ with $(r, \theta)$ restricted
to a region $\mathcal{S}_1$ of the form 
\begin{align} 
  \mathcal{S}_1 = \{(r, \theta) \in \mathcal{R} \, : \,
  0<\theta_1^*\leq |\theta| \leq \theta_0^*\leq \tfrac{\pi}{n}\}
\end{align} 
where $\theta_0^*, \theta_1^*$ are any positive constants.

To see what happens in the remainder of $\mathcal{R}$, we now turn to
analyzing $L\circ S_\alpha^\lambda$ as $\lambda \rightarrow \infty$
for $\alpha>0$ fixed. By fixing the constant $\theta_1^*>0$ from the
definition of $\mathcal{S}_1$ above to be sufficiently small, it is
reasonable to assume that the dominant behavior of $L\circ
S_\alpha^\lambda$ in $\lambda$ can be discovered in
$\mathcal{R}\setminus \mathcal{S}_1$ by considering the power series
expansion of the coefficients of $L$. Fixing a $J > \frac{n}{2}
+6$, we write
\begin{equation*}
   L = r\partial_r + n\theta \partial_{\theta}+
  \P(r, \theta)\partial_r + \Q(r, \theta) \partial_{\theta} +
  \frac{\sigma^2}{2 r^n }\partial_{r}^2 +
  \frac{\sigma^2}{2r^{n+2}}\partial_{\theta}^2
\end{equation*}
with 
\begin{equation}
\label{eqn:exprPQ}
\begin{aligned}
  \P(r, \theta) &=  \sum_{i=1}^{J-1} \alpha_{i}r\theta^i +
  \sum_{i = 0}^{n+1} \sum_{j=0}^{J-1} \beta_{ij} r^{-i} \theta^j + R_\P (r,\theta)\\
 \Q(r, \theta) &= \sum_{i=1}^{n+1} \gamma_i r^{-i} +
\sum_{i=2}^{J-1} \delta_{i} \theta^i + \sum_{i=1}^{n+2}\sum_{ j =1}^{J-1}
\epsilon_{ij} r^{-i} \theta^j + R_\Q(r,\theta)
\end{aligned}
\end{equation}
where $\alpha_i, \beta_{ij}, \gamma_i, \delta_i, \epsilon_{ij}$ are
constants and the remainder functions $R_\P$
and $R_\Q$ satisfy
\begin{align}
&|R_\P(r,\theta)| \leq C_\P (r+1) |\theta|^{J}, \quad |R_\Q(r,\theta)| \leq C_\Q |\theta|^{J}, \quad  J> \frac{n}{2}+6,  
\end{align}
for some positive constants $C_\P, C_\Q$. We have switched from $P$
and $Q$ to $\P$ and $\Q$ because, in $\P$ and $\Q$, we include higher
order terms from the power series expansion of $r\cos(n \theta)$ and
$\sin(n \theta)$, respectively.

We begin by considering the region just next to $\mathcal{S}_1$; that
is, we analyze $L\circ S_\alpha^\lambda$ as $\lambda \rightarrow
\infty$ when $\alpha >0$ is fixed and small.  Looking at $L\circ
S_\alpha^\lambda$ for $\lambda >0$ large, the following four terms are
candidates for any dominant balance of $L$ as $r\rightarrow \infty$:
\begin{align}
   r \partial_r + n \theta \partial_\theta +\sum_{i=1}^{\lfloor \frac{n}2 \rfloor +1 }
  \lambda^{\alpha-i}\gamma_i r^{-i} \partial_\theta + \lambda^{2\alpha
    - (n+2)} \frac{\sigma^2}{2 r^{n+2}} \partial_{\theta}^2= (I)+(II)+\sum_{i=1}^{\lfloor \frac{n}2 \rfloor +1 }(III_i)+(IV) \label{eq:I-II-III-IV}
\end{align}
where $\lfloor x \rfloor$ is the greatest integer less than or equal
to $x$.  Note that we have neglected the
$\delta_i\theta^i \partial_\theta$, $i\geq 2$, terms since for
$|\theta|$ small they are dominated by $n \theta \partial_\theta$.
Similarly, we have neglected all of the $\epsilon_{ij} r^{-i}\theta^j \partial_\theta$
terms since for $\theta$ small the corresponding $(III_i)=\gamma_i
r^{-i} \partial_\theta$ term always dominates it. We have also
neglected all of the $\alpha_{i}r\theta^i\partial_r$ and $\beta_{ij}
r^{-i} \theta^j \partial_r$ terms since they are always dominated by
the $r \partial_r$ term for $r$ large and $\theta$ small.  The terms
$(III_i)$ must be included since there is always a region, dictated by
the value of $\alpha >0$, where $\theta$ is small enough so that $(II)$
is dominated by some collection of the $(III_i)$ as $r\rightarrow
\infty$.

It is also important to realize why we have truncated the sum $\sum
(III_i)$ at $i=\lfloor \frac{n}2 \rfloor +1$.  Comparing the diffusion
term $(IV)$ with the terms $(III_i)$, observe that we need only
consider indices $i$ of $(III_i)$ satisfying $\alpha - i \geq 2 \alpha
-(n+2)$.  Rearranging this conditions produces the restriction $i \leq
n+2 -\alpha$. To obtain the claimed condition $i\leq \lfloor \frac{n}2
\rfloor+1$, we must first understand the relevant range of
$\alpha$. When $2\alpha -(n+2) =0$ the term $(II)$ balances the term
$(IV)$. Solving this condition to find $\alpha = \frac{n}2 +1$ and
substituting this value of $\alpha$ into $i \leq n+2 -\alpha $, we
obtain the claimed bound $i\leq \lfloor \frac{n}2 \rfloor +1$.

Assumption~\ref{I-assumption:nlot} from \cite{HerzogMattingly2013I}
translated to this context implies that $\gamma_i=0$ for $i \in
\{1,\dots \lfloor \frac{n}2\rfloor +1\}$; and hence in the case
considered previously, none of the $(III_i)$ terms we have retained in
\eqref{eq:I-II-III-IV} are present.  To further illustrate the
differences encountered here, we now analyze $L\circ S_\alpha^\lambda$
as $\lambda \rightarrow \infty$ for $\alpha >0$ fixed in the relevant
range $(0, \lfloor\frac{n}2\rfloor +1]$.

For $0<\alpha<1$, observe that 
\begin{align*}
(I)+(II) \gg \sum_{i=1}^{\lfloor \frac{n}2 \rfloor+1} (III_i)+(IV)\,\, \text{ as }\,\, \lambda \rightarrow \infty.
\end{align*}
When $\alpha=1$, however, as  $\lambda \rightarrow \infty$ we see that 
\begin{equation}\label{balanceIII_1}
(I)+(II)+(III)_1 \gg \sum_{i=2}^{\lfloor n/2\rfloor +1} (III_i) + (IV) \,.
\end{equation} 
Hence we expect
\begin{align}
T_2 = r\partial_r +  n\theta \partial_\theta
\end{align}
to satisfy $L \approx T_2$ as $r\rightarrow \infty$ when all paths to
infinity are restricted to a region of the form
\begin{align*}
\{(r, \theta) \in \mathcal{R}\, : \, r\geq r^*, \, b(r)\leq |\theta| \leq \theta_1^*  \}
\end{align*}
where $\theta_1^* >0$ is small enough and 
\begin{align}\label{bBoundaryRestriction}
b(r)= c r^{-1}+o(r^{-1})\,\,\text{ as }\,\, r\rightarrow \infty
\end{align}
for some large constant $c>0$.  Note that $c>0$ is chosen to be large
to assure that the term $(III)_1$ is not also dominant in the region
defined above.  We also leave open the choice of a specific curve $b$
because what will happen in the remaining part of $\mathcal{R}$:
\begin{align*}
\{(r, \theta) \in \mathcal{R}\, : \, r\geq r^*, \, |\theta| \leq b(r)  \}
\end{align*}     
will suggest its definition.

When $\alpha=1$, $(III)_1$ also becomes dominant suggesting that
\begin{equation}\label{eq:balanceOP-1}
r \partial_r + n\theta \partial_\theta + \gamma_1 r^{-1} \partial_\theta
\end{equation}
should be the asymptotic operator in the next region.  However if
$\gamma_1 \neq 0$, then on the curve
$n\theta = -\gamma_1 r^{-1}$
\begin{align*}
n\theta \partial_\theta + \gamma_1 r^{-1} \partial_\theta = 0.  
\end{align*} 
and all of the $\partial_\theta$ terms in \eqref{eq:balanceOP-1}
vanish. Hence we must turn to the terms neglected above and do a further analysis of $L\circ S_\alpha^\lambda$ as $\lambda \rightarrow \infty$ to find the dominant $\partial_\theta$
term. In fact, it is likely that the dominant balance expressed in
\eqref{eq:balanceOP-1} will fail to hold before the terms above exactly cancel.

To help see which terms need to be included in a neighborhood of the
curve defined by $n\theta = -\gamma_1 r^{-1}$, we make a convenient
change of coordinates. The basic idea is to remove the term $(III)_1$
by means of a coordinate transformation, returning us to a setting
like that considered above when $\alpha \in (0,1)$.  Introducing the
mapping $(r, \theta) \mapsto (r, \phi_3)$ where $\phi_3$ is defined by
\begin{equation}
\phi_3 =r \theta + \frac{\gamma_1}{n+1}, 
\end{equation}
we see that the operator $L$
transforms to
\begin{align}
\label{Lc3}
L_{(r, \phi_3)} = r \,\partial_{r} + (n+1)
\phi_3\, \partial_{\phi_3} + \P_3 \,\partial_{r} +
\Q_3\, \partial_{\phi_3} + \frac{\sigma^2}{2
  r^n} \partial_{\phi_3}^2 +
\Big(\frac{\sigma^2}{2r^n}\partial_r^2 \Big)_{(r, \phi_3)}
\end{align}
where
\begin{equation}
\label{eqn:PQ3}
\begin{aligned}
  \P_3(r,\phi_3) &=\sum_{i=0}^{n+J}\sum_{ j=0}^{J-1} \alpha_{ij}^{(3)} r^{-i} \phi_3^j + R_{P_3}\\
  \Q_3(r,\phi_3) & = \sum_{i =1}^{n+J+1} \gamma_i^{(3)} r^{-i} +
  \sum_{i=1}^{n+J+1}\sum_{j= 1}^{J}\beta_{ij}^{(3)} r^{-i} \phi_3^j +
  R_{Q_3},
\end{aligned}
\end{equation}
$\alpha_{ij}^{(3)}, \gamma_i^{(3)}, \beta_{ij}^{(3)}$ are constants and, because $|\theta| \leq \theta_1^* \leq C$, the
remainders $R_{P_3}$, $R_{Q_3}$ satisfy
\begin{align*}
  |R_{P_3}(r,\phi_3)| &\leq C_{P_3} (r +1) [|r^{-1} \phi_3|^J + r^{- J}]\\
  |R_{Q_3}(r,\phi_3)|&\leq C_{Q_3} (r +1)[|r^{-1} \phi_3|^{J} + r^{-J}].
\end{align*}
We have chosen not to write out the term
$\tfrac{\sigma^2}{2r}\partial_r^2$ in \eqref{Lc3} in the variables
$(r, \phi_3)$ because it is too long of an expression and since it
is always dominated, by considering the appropriate scaling
transformation, by the other terms in $L_{(r, \phi_3)}$ as $r
\rightarrow \infty$.

After this change of variables, note that $(II)+(III)_1$ has
transformed into $(n+1)\phi_3 \partial_{\phi_3}$, hence we have
``removed" $(III)_1$.  However, note that a new $\gamma_1^{(3)}
r^{-1}$ term is generated, playing the same role as $(III)_1$ did in
the previous coordinate system.  While this may not seem like
progress, notice that angular diffusion term (analogous to $(IV)$) now
has a coefficient $r^{-n}$ where it was $r^{-n-2}$ in the old
coordinates. Hence this term is more powerful.  By a similar line of
reasoning to the above, only the terms analogous to $(III_i)$ with $i
\in \{1,\ldots, \lfloor \frac{n}2 \rfloor\}$ are not dominated by the
noise. This is one less than previously. Therefore by performing such
substitutions iteratively, we will be able to remove enough terms so
that in the final coordinate system, the angular diffusion term will
dominate all analogous terms to the $(III_i)$'s . An important point
which makes this iteration possible is that $\P_3$ and $\Q_3$ have
that same forms as $\P$ and $\Q$, respectively, in that the lower limits
of the sums do not change. Even though the upper limits of the sums will increase, these added contributions are of lower order so they do not change the analysis. 

To finish the analysis in the variables $(r, \phi_3)$, we need to complete our
understanding of the boundary $|\theta|=b(r)$, extract the dominant operator
which replaces $T_2$ after we cross this boundary, and determine the
lower limit of the region where this new operator remains dominant.

To do this, we again consider $L_{(r, \phi_3)}$ under the scaling transformation
$S_\alpha^\lambda(r, \phi_3):= (\lambda r, \lambda^{-\alpha}
\phi_3)$.  First, we note that when $\alpha=0$
\begin{align*}
  L_{(r, \phi_3)}\circ S_0^\lambda = r \partial_{r} + (n+1)
  \phi_3 \partial_{\phi_3} + o(1)\,\, \text{ as }\,\, \lambda
  \rightarrow \infty
\end{align*}   
implying that 
\begin{align}
T_3 = r \partial_{r} + (n+1) \phi_3 \partial_{\phi_3}
\end{align}
satisfies $L_{(r, \phi_3)}\approx T_3$ as $r\rightarrow \infty$
when paths to infinity are restricted to a region where $|\phi_3|$ is
bounded and bounded away from zero. Thus we choose the second
region to be
\begin{align}
  \mathcal{S}_2 = \{(r, \theta) \in \mathcal{R}\, : \, r\geq r^*, \,
  |\phi_3| \geq \phi^*,\, |\theta| \leq \theta_1^* \}.
\end{align}
Notice that this choice of boundary $|\phi_3| =\phi^*$  is consistent the previous
requirement on the boundary function $b(r)$ in the $(r,\theta)$
variables given in \eqref{bBoundaryRestriction} provided $\phi^* > \gamma_1/(n+1)$.  In a subset of the region $|\phi_3| < \phi^*$, the approximation $L_{(r, \phi_3)}\approx T_3$  holds as $r\rightarrow \infty$. To discover the boundary of this region, we now study $L_{(r, \phi_3)}\circ
S_\alpha^\lambda $ for $\alpha>0$.

As before in the previous coordinate system, there are four terms
which are potentially involved in any dominant balance  of the terms in $L_{(r, \phi_3)}\circ S_\alpha^\lambda $
as $\lambda\rightarrow \infty$:
\begin{align*}
  r \partial_{r} + (n+1)\phi_3 \partial_{\phi_3} +
  \lambda^{\alpha-1}\gamma_1^{(3)} r^{-1} \partial_{\phi_3} +
  \lambda^{2\alpha - n} \frac{\sigma^2}{2 r^n} \partial_{\phi_3}^2 = (I)_{3}+(II)_3+(III)_3+(IV)_3.
\end{align*}  
Notice that the terms $(I)_3, (II)_3, (III)_3, (IV)_3$ are completely
analogous the terms $(I),(II),(III)_1, (IV)$ from the preceding discussion.

As already noted, after the change to the $(r,\phi_3)$ coordinates,
only terms  $\gamma_i^{(3)} r^{-i}$ with $i \in \{1,\ldots, \lfloor \frac{n}2 \rfloor\}$ could
possibly dominate or balance the angular diffusion term $(IV)_3$, and the
term $(III)_3$ is the leading order term of this form.   Hence when
$n=1$ or $n=2$, there is only one such dominant term of this form, namely  $(III)_3$, and it is of either the same (case $n=2$) or lesser (case $n=1$) order as $(IV)_3$. In general $(n\geq 3$), we will have to perform additional transformations to remove all
of the possibly dominant terms.  Before considering general $n\geq 3$, we pause to
finish the analysis in the cases when $n=1$ and $n=2$.

\subsubsection{Remaining operators and regions when  $n=1$}

For every $\alpha\geq 0$, we see that 
\begin{align*}
(I)_3+(II)_3 +(IV)_3 \gg (III)_3 \, \, \, \text{as} \, \, \, \lambda \rightarrow \infty.  
\end{align*}
If $0\leq \alpha < \frac{1}{2}$, then  
\begin{align*}
(I)_3 + (II)_3 \gg (IV)_3\, \, \, \text{as} \, \, \, \lambda \rightarrow \infty.  
\end{align*}
When $\alpha=1/2$, the term $(IV)_3$ also becomes dominant in
$\lambda$.  In particular, the region where we expect $ L_{(r,
  \phi_3)}\approx T_3$ as $r \rightarrow \infty$ is precisely
\begin{align}
\mathcal{S}_{3} = \{(r, \theta)\in \mathcal{R} \, :\, r\geq r^*,
\,\eta^*  r^{-\frac{1}{2}} \leq |\phi_3| \leq \phi^*, \, |\theta|
\leq \theta_1^*   \}, 
\end{align} 
for some $\eta^*>0$.  Additionally, the operator
\begin{align}
A = r \partial_{r} + 2  \phi_3 \partial_{\phi_3} + \frac{\sigma^2}{2 r} \partial_{\phi_3}^2
\end{align} 
contains the dominant part of $L_{(r, \phi_3)}$ in the region
\begin{align}
\mathcal{S}_{4}= \{(r, \theta)\in \mathcal{R} \, : \, r \geq r^*, \,
|\phi_3|\leq \min( \eta^*   r^{-1/2},\phi^*), \, |\theta|\leq
\theta_1^* \}. 
\end{align}
Summing this up, we have seen that when $n=1$, the approximating
operators are $T_1, T_2, T_3, A$ with corresponding regions
$\mathcal{S}_1, \mathcal{S}_2, \mathcal{S}_{3}, \mathcal{S}_{4}$ where
we expect the approximation to be valid for $r>0$ large.

\begin{remark}
\label{rem:parameterchoice}
We have already introduced a number of parameters (e.g. $\theta_1^*$,
$\phi^*$, $\eta^*$, $r^*$) thus far that will have to be chosen to
satisfy a number of properties.  Instead of writing these properties
explicitly, we simply need to make sure that we vary the parameters in
a consistent way to obtain them.  That is, we will always choose
$\theta_1^*>0$ small enough, then pick $\phi^*=\phi^*(\theta_1^*)>0$
large enough, then choose $\eta^*=\eta^*(\theta_1^* , \phi^*)>0$ large
enough, and then finally pick $r^*= r^*(\theta_1^*, \phi^*, \eta^*)>0$
large enough.  For example, to assure that $\mathcal{S}_3$ and
$\mathcal{S}_4$ defined above are of the required form outlined in
Section \ref{I-subsec:Bas_structI} of \cite{HerzogMattingly2013I}, we can choose the parameters $\theta_1^*$,
$\phi^*$, $\eta^*$, and $r^*$ in this way to see that in fact
\begin{align*}
  &\mathcal{S}_{3} = \{(r, \theta)\in \mathcal{R} \, :\, r\geq r^*,
  \,\eta^* r^{-\frac{1}{2}} \leq |\phi_3| \leq \phi^* \},\\
  &\mathcal{S}_{4}= \{(r, \theta)\in \mathcal{R} \, : \, r \geq r^*,
  \, |\phi_3|\leq \eta^*  r^{-1/2} \}.
\end{align*}     
\end{remark}

\subsubsection{Remaining operators and regions when  $n =  2$.}  
Notice that for $0\leq \alpha < 1$
\begin{align*}
(I)_3+(II)_3\gg (III)_3+(IV)_3\,\, \, \text{ as }\, \, \,\lambda \rightarrow \infty .  
\end{align*}
When $\alpha=1$, then $(III)_3+(IV)_3$ also becomes dominant in
$\lambda$.  Therefore, this suggests that the region where $T_3\approx
L_{(r, \phi_3)}$ as $r \rightarrow \infty$ is of the form
\begin{align}
\mathcal{S}_{3} = \{(r, \theta)\in \mathcal{R} \, : \, r \geq r^*, \,
\eta^* r^{-1}\leq |\phi_3| \leq \phi^* \} 
\end{align}
for some $\eta^*>0$.  Here again, we have picked the parameters in the
way discussed in Remark \ref{rem:parameterchoice}.  Notice also that
the operator
\begin{align}
A = r \partial_{r} + 3\phi_3 \partial_{\phi_3} + \gamma_1^{(1)}
r^{-1} \partial_{\phi_3} +
\frac{\sigma^2}{2 r^2} \partial_{\phi_3}^2   
\end{align}    
contains the dominant part of $L_{(r, \phi_3)}$ as $r\rightarrow
\infty$ in the region
\begin{align}
  \mathcal{S}_{4} = \{(r, \theta)\in \mathcal{R} \, : \, r \geq r^*,
  \, |\phi_3|\leq \eta^* r^{-1} \}
\end{align}
where we have again picked $\phi^*$ and $r^*$ according to Remark
\ref{rem:parameterchoice}.  Summing this up, we have seen that when
$n=2$, the asymptotic operators are $T_1, T_2, T_3, A$ with
corresponding regions $\mathcal{S}_1, \mathcal{S}_2, \mathcal{S}_{3},
\mathcal{S}_{4}$ where the approximation is expected to be valid.

\subsection{Remaining analysis when $n=3,4$.}
If $0 \leq \alpha <1$, then
\begin{align*}
(I)_3+(II)_3 \gg (III)_3+(IV)_3 \, \, \, \text{ as } \, \, \, \lambda \rightarrow \infty.
\end{align*}
Therefore, we have that $L_{(r, \phi_3)} \approx T_3$ as
$r\rightarrow \infty$ in some region of the form
\begin{align}
\label{eqn:potlowbs3}
\{(r, \theta)\in \mathcal{R} \, : \, r \geq r^*, \, b(r)\leq |\phi_3| \leq \phi^* \}
\end{align}  
where $b$ satisfies
\begin{align*}
b(r) = c r^{-1} + o(r^{-1}) \,\, \text{ as } \,\, r \rightarrow \infty  
\end{align*}
for some $c>0$.  If $\alpha\geq 1$ however, it is not clear if the
terms in $L_{(r, \phi_3)}$ corresponding to
$(I)_3+(II)_3+(III)_3+(IV)_3$ contain the dominant part of the
operator because
\begin{align*}
(n+1) \phi_3 \partial_{\phi_3} + \gamma_1^{(3)} r^{-1} \partial_{\phi_3} = 0
\end{align*}  
when $(n+1) \phi_3 =- \gamma_1^{(3)}r^{-1}$.  Hence to analyze
$L_{(r, \phi_3)}$ around this other potentially dangerous curve, we
make another substitution $(r, \phi_3) \mapsto (r, \phi_4)$ where
\begin{align*}
\phi_4 &= r \phi_3 + c_3.
\end{align*} 
and $c_3 = \frac{\gamma_1^{(3)}}{n+2}$.  As before, we use the new variables $(r, \phi_4)$ to define the boundary curve $b$ precisely by setting  
\begin{align}
  \mathcal{S}_3 = \{(r, \theta)\in \mathcal{R} \, : \, r \geq r^*, \,
  |\phi_4|\geq \phi^*,\,|\phi_3| \leq \phi^* \}.
\end{align}

Now write $L_{(r, \phi_3)}$ in the variables $(r, \phi_4)$ to see that 
\begin{align*}
  L_{(r, \phi_4)} = r \partial_{r}+ (n+2) \phi_4 \partial_{\phi_4}+ \P_4 \partial_{r} + \Q_4\partial_{\phi_4} +
  \frac{\sigma^2}{2 r^{n-2}}\partial_{\phi_4}^2 +
  \bigg(\frac{\sigma^2}{2 r^n} \partial_r^2 \bigg)_{(r, \phi_4)}
\end{align*}  
where
\begin{align*}
  \P_4 &= \sum_{i\geq 0, j\geq 0} \alpha_{ij}^{(4)} r^{-i} \phi_4^{j} + R_{P_4}\\
  \Q_4&= \sum_{i\geq 1}\gamma_i^{(4)} r^{-i} +
  \sum_{i\geq 1, j \geq 1} \beta_{ij}^{(4)} r^{-i} \phi_4^j +
  R_{Q_4}
\end{align*}
where $\alpha_{ij}^{(4)}, \gamma_i^{(4)}, \beta_{ij}^{(4)}$ are
constants, all sums above are finite sums and, since $\phi_4$ will be
bounded by $\phi^*$ in any subsequent region, $R_{P_4}$ and $R_{Q_4}$ satisfy
\begin{align*}
|R_{P_4} | &\leq C_{P_4}(r +1) [|r^{-2} \phi_4|^{J} + r^{-J}] \\
|R_{Q_4} | & \leq C_{Q_4}(r^{2} +1) [|r^{-2} \phi_4|^{J} + r^{-J}]
\end{align*}
for some positive constants $C_{P_4}, C_{Q_4}$.  Here, note that both $C_{P_4}$ and $C_{Q_4}$ can be chosen independent of $\phi^*$ by picking $r^* > \phi_4^*$.    Considering the
effect of $L_{(r, \phi_4)}$ under $S_\alpha^\lambda(r, \phi_4):=
(\lambda r, \lambda^{-\alpha} \phi_4)$, $\alpha\geq 0$, we again
consider the following four terms in $L_{(r, \phi_4)}\circ S_\alpha^\lambda$ which could become dominant in $\lambda$ as $\lambda \rightarrow \infty$:
\begin{align*}
  &r \partial_{r}+(n+2) \phi_4 \partial_{\phi_4} +
  \lambda^{\alpha-1} \gamma^{(4)}_1 r^{-1} \partial_{\phi_4}+
  \lambda^{2\alpha-(n-2)} \frac{\sigma^2}{2
    r^{n-2}} \partial_{\phi_4}^2\\ 
  & =(I)_4 + (II)_4+(III)_4+(IV)_4 .
\end{align*} 
Similarly, we can now uncover all asymptotic operators and their associated
regions when $n=3$ and $n=4$ because the angular noise term $(IV)_4$ is of sufficient strength in $\lambda$.

\subsubsection{Remaining operators and regions when $n=3$}
Analogous to the case when $n=1$,
\begin{align}
T_4 = r \partial_{r} + 5 \phi_4 \partial_{\phi_4} 
\end{align}
satisfies $L_{(r, \phi_4)}\approx T_4$ as $r\rightarrow \infty$
when paths to infinity are restricted to a region of the form
\begin{align}
  \mathcal{S}_{4}=\{(r, \theta)\in \mathcal{R} \, : \, r \geq r^*, \,
  \,\eta^* r^{-\frac{1}{2}} \leq |\phi_4| \leq \phi^* \}
\end{align}
where the parameters $\eta^*, \phi^*>0$ have been chosen according to
Remark \ref{rem:parameterchoice}.  Also,
\begin{align}
  A= r \partial_{r} + 5 \phi_4 \partial_{\phi_4} +
  \frac{\sigma^2}{2 r} \partial_{\phi_4}^2
\end{align}
can be used to approximate $L_{(r, \phi_4)}$ asymptotically for
large $r$ in a region of the form
\begin{align}
  \mathcal{S}_{5}=\{(r, \theta)\in \mathcal{R} \, : \, r \geq r^*, \,
  \, |\phi_4| \leq \eta^* r^{-\frac{1}{2}}\}
\end{align}
where we have again picked the parameters as described in Remark
\ref{rem:parameterchoice}.  Thus, when $n=3$, we obtain the
approximating operators $T_1, T_2, T_3, T_4, A$ with corresponding
regions $\mathcal{S}_1, \mathcal{S}_2, \mathcal{S}_3, \mathcal{S}_4,
\mathcal{S}_5$ where the approximation is expected to be valid.

\subsubsection{Remaining operators and regions when $n= 4$}  Similar to the case when $n=2$, the region where 
\begin{align*}
T_4 = r \partial_{r} + 6 \phi_4 \partial_{\phi_4}
\end{align*}
is in good approximation to $L_{(r, \phi_4)}$ for $r>0$ large is given by 
\begin{align*}
  \mathcal{S}_4 = \{(r, \theta) \in \mathcal{R} \, : \, r \geq
  r^*,\,\eta^* r^{-1} \leq |\phi_4| \leq \phi^* \}
\end{align*}
where the parameters have been chosen appropriately.  Also,
\begin{align*}
  A = r \partial_{r} + 6 \phi_4 \partial_{\phi_4} + \gamma_1^{(4)}
  r^{-1} \partial_{\phi_4}+ \frac{\sigma^2}{2
    r^2} \partial_{\phi_4}^2
\end{align*}
contains the dominant, large $r$ behavior corresponding to $L_{(r,
  \phi_4)}$ in
\begin{align*}
\mathcal{S}_5 = \{(r, \theta) \in \mathcal{R} \, : \, r \geq r^*,\,|\phi_4|\leq \eta^* r^{-1} \}.  
\end{align*}
Thus when $n=4$, we obtain the asymptotic operators $T_1, T_2, T_3,
T_4, A$ and their regions $\mathcal{S}_1, \mathcal{S}_2,
\mathcal{S}_3, \mathcal{S}_4, \mathcal{S}_5$ of approximation.

\subsection{All operators and regions for general $n\geq 1$:}
\label{subsec:ingeneral}
We continue until this inductive procedure until it stops.  More precisely, if $n=2j+1$ or
$n=2j+2$ for some $j\geq 0$, then the analysis yields the
asymptotic operators
\begin{align*}
T_1, T_2, \ldots, T_{j+3}, A
\end{align*}
and respective regions 
\begin{align*}
  \mathcal{S}_1, \mathcal{S}_2, \ldots,\mathcal{S}_{j+3},
  \mathcal{S}_{j+4}.
\end{align*}
To write each of them explicitly, set $\phi_2:=\theta$ for $m\geq 3$ and let 
\begin{align}
\label{eqn:angularcoordinate}
  \phi_{m}= r  \phi_{m-1} +c_{m-1}.
\end{align}   
where $c_2 =\frac{\gamma_1}{n+1}$ and $c_m = \frac{\gamma_1^{(m)}}{n+m-1}$ for $m \geq 3$.  We see that $T_1, \ldots, T_{j+3}$ are given by  
\begin{alignat*}{2}
&T_1 = r\cos(n\theta) \partial_r + \sin(n \theta) \partial_\theta &&\\
\nonumber &T_m  = r \partial_{r} + (n+{m-2}) \phi_{m} \partial_{\phi_{m}},\,\,\,  && m=2, 3,  \ldots, j+3.
\end{alignat*}
If $n=2j+1$, then the diffusive operator $A$ satisfies
\begin{align*}
  A = r \partial_{r} +(3j+2)
  \phi_{j+3} \partial_{\phi_{j+3}} + \frac{\sigma^2}{2
    r} \partial_{\phi_{j+3}}^2.
\end{align*}
On the other hand if $n=2j+2$, we have 
\begin{align*}
A=
r \partial_{r} +(3j+3) \phi_{j+3} \partial_{\phi_{j+3}} +\gamma_1^{(j+3)} r^{-1} \partial_{\phi_{j+3}} +\frac{\sigma^2}{2 r^2} \partial_{\phi_{j+3}}^2. 
\end{align*}
Choosing the parameters $\theta_1^*, \phi^*, \eta^*, r^*$ according to
Remark \ref{rem:parameterchoice}, we may write all corresponding
regions as follows. Note first that $\mathcal{S}_1, \mathcal{S}_2,
\ldots, \mathcal{S}_{j+2}$ are given by
\begin{align*}
  &\mathcal{S}_1 = \{(r, \theta) \in \mathcal{R}\, : \,r\geq r^*, \,
  0<\theta_1^*\leq |\theta| \leq \theta_0^*  \}\\
  &\mathcal{S}_2= \{ (r, \theta) \in \mathcal{R}\, : \, r\geq r^*,\, |\phi_3|\geq \phi^*,\, |\theta| \leq \theta_1^*\} \\
  \nonumber &\mathcal{S}_m = \big\{(r, \theta) \in \mathcal{R}\, :
  \,r\geq r^*, |\phi_{m+1}|\geq \phi^*, \, |\phi_m|\leq \phi^* \big\}
\end{align*}
for $m=3, \ldots, j+2$. If $n=2j+1$, the final two regions satisfy  
\begin{align*}
&\mathcal{S}_{j+3}= 
\big\{(r, \theta) \in \mathcal{R}\, : \,r\geq r^*, \,  \eta^*
r^{-\frac{1}{2}}\leq |\phi_{j+3}| \leq \phi^*  \big\}\\ 
&\mathcal{S}_{j+4}= 
\big\{(r, \theta) \in \mathcal{R}\, : \,r\geq r^*, \,  |\phi_{j+3}|\leq \eta^* r^{-\frac{1}{2}} \big\}
\end{align*}
On the other hand if $n=2j+2$, then $\mathcal{S}_{j+3}$ and $\mathcal{S}_{j+4}$ are given by 
\begin{align*}
  &\mathcal{S}_{j+3}=\big\{(r, \theta) \in \mathcal{R}\, :\, r\geq
  r^*, \,  \eta^* r^{-1} \leq |\phi_{j+3}|\leq \phi^* \big\}\\ 
  &\mathcal{S}_{j+4}=\big\{(r, \theta) \in \mathcal{R}\, :\, r\geq
  r^*, \, |\phi_{j+3}| \leq \eta^* r^{-1} \big\}.
\end{align*}

It is also important to notice that $L$, when written in the variables
$(r, \phi_m)$ for $m=3, \ldots,j+3$, satisfies
\begin{align}
\label{eqn:asyLm}
L_{(r, \phi_m)} =  r  \partial_{r} + (n+m-2) \phi_m \partial_{\phi_m}+ \P_m \partial_{r} + \Q_{m} \partial_{\phi_m}
+\frac{\sigma^2}{2 r^{n-2m+6}} \partial_{\phi_m}^2+
\bigg(\frac{\sigma^2}{2 r^n} \partial_r^2\bigg)_{(r, \phi_m)}
\end{align}
where
\begin{align*}
  \P_m &=\sum_{i, j\geq 0} \alpha_{ij}^{(m)} r^{-i} \phi_m^{j} + R_{P_m}\\
  \Q_m &= \sum_{i\geq 1} \gamma_{i}^{(m)} r^{-i} +
  \sum_{i\geq 1, j \geq 1} \beta_{ij}^{(m)} r^{-i} \phi_m^j+ R_{Q_m}
\end{align*}
where $\alpha_{ij}^{(m)}, \gamma_i, \beta_{ij}^{(m)}$ are constants, all sums are finite sums, and by the choice of $J>\frac{n}{2}+6$ the remainders satisfy the following bounds on $\mathcal{S}_m$
\begin{align*}
|R_{P_m}|\leq C_{P_m}r ^{-2}, \qquad |R_{Q_m}| \leq C_{Q_m} r^{-2},  
\end{align*}
$m=3, \ldots, j+4$.  Note that the constants $C_{P_m}$ and $C_{Q_m}$, $m\geq 3$, depend on $\phi^*$ but they do not depend on $r^*$.

\begin{figure}
\centering
\begin{tikzpicture}[scale=1.1]

\draw[thick,color=black,opacity=.75] (-6,2) -- (6,2) node[right] {$\theta$};
\draw[thick, color=gray] (0, 2) -- (0, 1.8)
node[below]{$0$};
\node[thick, color = black, opacity=.75] at (-6.2,8) {$r$};

 \draw[thick,color=blue] ({-6},2.2) -- ({-6},1.8)
 node[below]{$-\theta_1^*$};
 \draw[thick,color=blue] ({6},2.2) -- ({6},1.8)
 node[below]{$\theta_1^*$};
 \draw[thick, color=blue] ({-5.8},2) -- ({-6.2},2);
\node[color=blue] at (-6.3,2) {$r^*$};

\draw[fill=gray!50!white]  plot[smooth, samples=300, domain=2:8] (9.5/\x, \x) -- 
plot[smooth, samples=300, domain=8:2] ({6}, \x );
\draw[fill=gray!50!white]  plot[smooth, samples=300, domain=2:8] ({-6}, \x ) -- 
plot[smooth, samples=300, domain=8:2] (-10.5/\x, \x);
\node[color=black] at (4,5) {$\mathcal{S}_2$};
\node[color=black] at (-4,5) {$\mathcal{S}_2$};

\draw[fill=gray!25!white]  plot[smooth, samples=300, domain=2:8]  (-10.5/\x, \x) -- 
plot[smooth, samples=300, domain=8:2] (-11/\x^2-.5/\x, \x) ;
\draw[fill=gray!25!white]  plot[smooth, samples=300, domain=2:8]  (9/\x^2-.5/\x,\x) -- 
plot[smooth, samples=300, domain=8:2]  (9.5/\x, \x);
\node[color=black] at (1.4,4) {$\mathcal{S}_3$};
\node[color=black] at (-1.6,4) {$\mathcal{S}_3$};


\draw[thin,color=blue,,pattern=grid,pattern color
  =blue!60,opacity=.75]  plot[smooth, samples=300, domain=2:8]   (-11/\x^2-.5/\x, \x) -- 
plot[smooth, samples=300, domain=8:2] (-5/\x^2.5-.5/\x-1/\x^2, \x);
\draw[thin,color=blue,,pattern=grid,pattern color=blue!60,opacity=.75]  plot[smooth, samples=300, domain=2:8]  (5/\x^2.5-.5/\x-1/\x^2,\x) -- plot[smooth, samples=300, domain=8:2]  (9/\x^2 -.5/\x,\x);
\node[color=black] at (.6,2.6) {$\mathcal{S}_4$};
\node[color=black] at (-1.3,2.6) {$\mathcal{S}_4$};

\draw[fill=white]  plot[smooth, samples=300, domain=2:8]   (-5/\x^2.5 -.5/\x -1/\x^2, \x) -- 
plot[smooth, samples=300, domain=8:2] (5/\x^2.5-.5/\x-1/\x^2, \x);
\node[color=black] at (-.4,2.4) {$\mathcal{S}_5$};

\draw[red, very thick, domain=2.6:8, samples=100]
plot ({-1/\x^2-.5/\x},{\x});
\draw[red, very thick, domain=2:2.2, samples=100]
plot ({-1/\x^2-.5/\x},{\x});

\draw[dashed, green!50!black, very thick, domain=2.6:8, samples=100]
plot ({-.005*\x^3-1/\x^2-.5/\x},{\x});
\draw[dashed, green!50!black, very thick, domain=2:2.2, samples=100]
plot ({-.005*\x^3-1/\x^2-.5/\x},{\x});

\draw[dashed, green!40!black, very thick, domain=2.6:8, samples=100]
plot ({.005*\x^3-1/\x^2-.5/\x},{\x});
\draw[dashed, green!40!black, very thick, domain=2:2.2, samples=100]
plot ({.005*\x^3-1/\x^2-.5/\x},{\x});

\draw[dashed, green!40!black, very thick, domain=2.6:8, samples=100]
plot ({.007*\x^3-1/\x^2-.5/\x},{\x});
\draw[dashed, green!40!black, very thick, domain=2:2.2, samples=100]
plot ({.007*\x^3-1/\x^2-.5/\x},{\x});

\draw[dashed, green!40!black, very thick, domain=2.6:8, samples=100]
plot ({-.007*\x^3-1/\x^2-.5/\x},{\x});
\draw[dashed, green!40!black, very thick, domain=2:2.2, samples=100]
plot ({-.007*\x^3-1/\x^2-.5/\x},{\x});

\draw[dashed, green!40!black, very thick, domain=2.6:8, samples=100]
plot ({-.01*\x^3-1/\x^2-.5/\x},{\x});
\draw[dashed, green!40!black, very thick, domain=2:2.2, samples=100]
plot ({-.01*\x^3-1/\x^2-.5/\x},{\x});

\draw[dashed, green!40!black, very thick, domain=2.6:8, samples=100]
plot ({.01*\x^3-1/\x^2-.5/\x},{\x});
\draw[dashed, green!40!black, very thick, domain=2:2.2, samples=100]
plot ({.01*\x^3-1/\x^2-.5/\x},{\x});

\end{tikzpicture}
\caption{Recalling that the constants \(\protect c_i\) were introduced below equation \eqref{eqn:angularcoordinate}, a sketch of the regions \(\mathcal{S}_i\), \(i=2,3,4,5\), is plotted when \(n=3\), \(c_2=1/2\), \(c_3=1\), \(\phi^*=10\), and \(\eta^*=5\).  A simple example of an operator which has this same region decomposition is $L= r\partial_r + (3 \theta +2 r^{-1}+ 5 r^{-2}) \partial_\theta + \frac{\sigma^2}{r^3} \partial_r^2 + \frac{\sigma^2}{r^{5}}\partial_\theta^2 $.  
In the absence of noise, the dynamics defined by \(r^3L\) explodes in finite time along the solid  trajectory splicing the interior of $\mathcal{S}_5$ in the figure above.  The formula of this unstable trajectory is given by the equation $\phi_4 = r^2\theta + \frac{r}{2} + 1=0.$  Moreover, away from this trajectory in the absence of noise, solutions along \(r^3L\) push away from this unstable trajectory, eventually exiting though one of the boundaries \(\theta=\pm \theta_1^*\).  The dashed curves plotted above are a few representative stable trajectories for the system corresponding to the operator $r^3 L$.  The general formula for these stable trajectories is given by $\theta = \phi_4(0) r^3 - \frac{1}{2r} - \frac{1}{r^2}$, $\phi_4(0) \in \R_{\neq 0}$, .}
\label{fig1}
\end{figure}
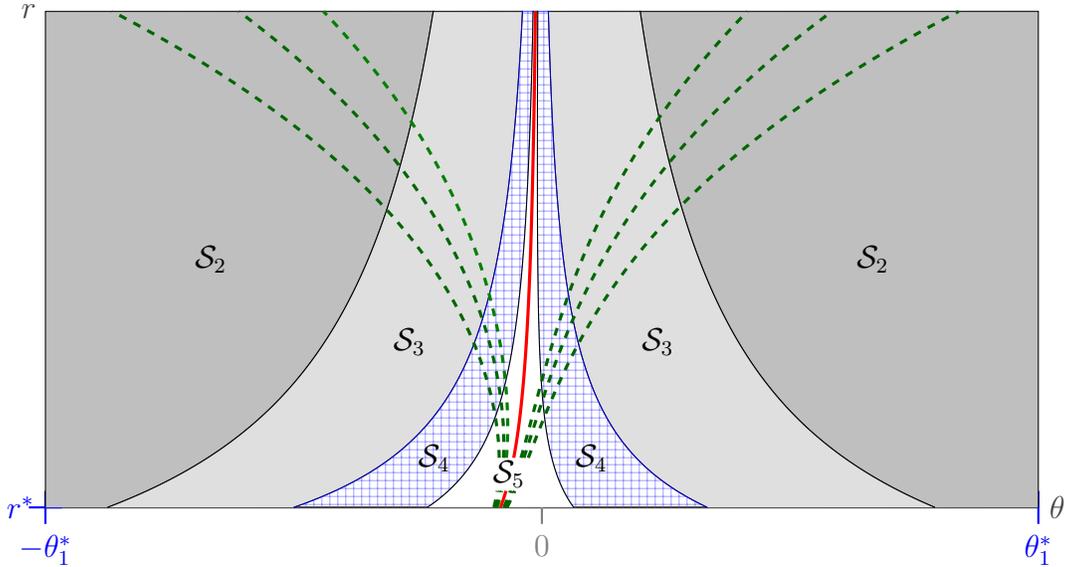

\section{The Construction of $\Psi$ on $\mathcal{R}$ in the general case}
\label{sec:constructionPsigen}

Employing the asymptotic analysis of the previous section, we now
define our candidate Lyapunov function $\Psi$ on the principal wedge
$\mathcal{R}$.  Recalling Section~\ref{I-sec:reductions} on Part I of
this work \cite{HerzogMattingly2013I}, we break up the
definition of $\Psi$ on $\mathcal{R}$ as follows
\begin{align*}
\Psi(r, \theta) = \begin{cases}
\psi_i(r, \theta) & \text{ if } (r, \theta) \in \mathcal{S}_i 
\end{cases}
\end{align*}  
where $i=0,1, \ldots, j+4$, $n=2j+1$ or $n=2j+2$.

As in Part I \cite{HerzogMattingly2013I}, to initialize the
propagation procedure used to define all of the
$\psi_i$'s we need one more region $\mathcal{S}_0$ (hence the $i=0$
above) defined by
\begin{align*}
  \mathcal{S}_0 = \{(r, \theta) \in \mathcal{R}\, : \, r\geq r^*, \,
  \theta_0^* \leq |\theta| \leq \tfrac{\pi}{n} \}
\end{align*}
where $\theta_0^*\in (\frac{\pi}{2n}, \frac{\pi}{n})$, and we define
the initial function $\psi_0$ on $\mathcal{S}_0$ by   
\begin{align*}
\psi_0(r, \theta) =r^p
\end{align*}
for some $p\in(0,n)$.  We recall that this choice is made because the
radial dynamics along $T_1$ is decreasing in  $\mathcal{S}_0$.

\subsection{The construction in the transport regions}
We first turn our attention to defining the functions $\psi_1, \psi_2,
\ldots, \psi_{j+3}$ respectively on the regions $\mathcal{S}_1,
\mathcal{S}_2, \ldots, \mathcal{S}_{j+3}$ as solutions to
boundary-value problems involving the asymptotic operators $T_1,
\ldots, T_{j+3}$.   Because the number of indices will be daunting
otherwise, we adopt the following conventions.
\begin{convention}
  When it is clear which coordinate system in which we are working,
  $(r, \phi_m)$ will be written more simply as $(r, \phi)$.  For
  example, $\psi_m(r, \phi_m)$ for $m=3,\ldots, j+3$ will often be
  written as $\psi_m(r, \phi)$.
\end{convention}  
\begin{convention}
  In the expressions we will derive for $\psi_1, \ldots, \psi_{j+3}$,
  there will be several parameters with double indices, e.g. see
  $p_{l,m}$ and $q_{l,m}$ below in Lemma \ref{lem:indtrans}.  The
  second index $m$ simply corresponds to the function $\psi_m$. Thus
  when it is clear that we are working with $\psi_m$, we will often
  write $p_{l, m}$ and $q_{l,m}$ more compactly as $p_l$ and $q_l$, 
  respectively.
\end{convention}

It is convenient in our analysis that the boundary conditions given for the Poisson equation defining $\psi_{j+4}$ on the
inner most region $\mathcal{S}_{j+4}$ (the one dominated by diffusion)
are symmetric under reflection in the angular coordinate
$\phi$ in $\mathcal{S}_{j+3}$. Thanks to this symmetry, the value of $\psi_{j+4}$ at the time
of exit of the diffusion from $\mathcal{S}_{j+4}$ depends only on the time of exit and
not on which side of the boundary it exits (each being
possible since the dynamics in $\mathcal{S}_{j+4}$ is diffusive).  As
we will see below, this allows us to define $\psi_{j+4}$ more simply.  Here, we will accomplish this desired symmetry by forcing the penultimate function $\psi_{j+3}$ to satisfy
\begin{align}
  \label{eqn:pensym}
  \psi_{j+3}(r, -\phi)= \psi_{j+3}(r, \phi)\,.
\end{align}
for $(r, \theta(r, \phi))\in \mathcal{S}_{j+3}$.  

The cost of producing this symmetry in the penultimate region can be seen in the
need to carefully choose the $h_i^\pm$ below in all of the preceding
regions since both the regions and the corresponding dominant
operators are inherently asymmetric in the angular coordinate.  Since we will have to make different choices of defining problems
above and below $\phi=0$ in each region to produce the symmetry, we will break up the definition of
$\psi_m$ in two pieces as follows
\begin{align*}
 \psi_m(r, \phi) = \begin{cases}
 \psi_m^+(r, \phi),&  (r, \theta(r, \phi)) \in \mathcal{S}_m,\, \, \phi>0\\
 \psi_m^-(r, \phi),& (r, \theta(r, \phi)) \in \mathcal{S}_m,\, \, \phi<0.  
 \end{cases}  
\end{align*}

\subsubsection*{The construction in $\mathcal{S}_1$} 
Let $\psi_1^{\pm}$ satisfy the following PDEs on $\mathcal{S}_1$:
\begin{equation}
\label{caseII:psi_1}
\left\{
\begin{aligned}
\big(T_1 \psi_1^{\pm}\big)(r, \theta) &= -h_1^{\pm} r^p | \theta|^{-q}\\
\psi_1^{\pm}(r, \pm \theta_0^*) &= \psi_0(r, \pm \theta_0^*).    
\end{aligned}\right.
\end{equation}   
where $q \in (\frac{p}{n}, 1)$ is fixed and $h_1^+, h_1^- >0$ will be
determined later (to produce the reflective symmetry).

Since $\theta_0^* > \frac\pi2$, we recall from Section \ref{I-sec:mot_lya} of Part I \cite{HerzogMattingly2013I} that the equations
above are not well-defined with the given boundary data because some
characteristics along $T_1$ cross $r=r^*$ before reaching the lines
$\theta=\pm\theta^*_0$. This can be easily remedied by enlarging the
domain of definition of the equation \eqref{caseII:psi_1} to
\begin{align*}
  \widetilde{ \mathcal{S}}_1=\Big\{(r, \theta) \in \mathcal{R}\, : \,
  0<\theta_1^*\leq |\theta|\leq \theta_0^*, \, r|\sin(n\theta_0^*)|^{\frac{1}{n}}\geq r^*\Big\}.
\end{align*}
With this modification of the domain, solving \eqref{caseII:psi_1} with  the method of
characteristics produces
\begin{align}
\label{eqn:psi_1c}
\psi_1^{\pm}(r, \theta) &= \frac{r^p}{|\sin(n\theta)|^{\frac{p}{n}}}
\bigg(|\sin(n\theta_0^*)| + h_1^\pm \int_{\theta}^{\pm\theta_0^*}
\frac{|\sin(n\alpha)|^{\frac{p}{n}}}{|\alpha|^q \sin(n\alpha)} \,
d\alpha \bigg)
\end{align}   
for $(r,\theta) \in \mathcal{S}_1$.    
In particular, we observe that $\psi_1$ is homogeneous under
$S_0^\lambda$, $\psi_1(r, \theta)>0$ for all $(r, \theta)$ with $r>0$
and $|\theta|\in (0,\tfrac{\pi}{n})$, and $\psi_1(r, \theta)
\rightarrow \infty$ as $r\rightarrow \infty$, $(r, \theta) \in
\mathcal{S}_1$ .

\subsubsection*{The construction in $\mathcal{S}_2$}
In a similar fashion, let $\psi_2^{\pm}$ solve
\begin{equation}
\label{eqn:psi_2}
\left\{\begin{aligned}
  \big(T_2 \psi_2^{\pm}\big)(r, \theta) &= - h_2^{\pm} r^p | \theta|^{-q}\\
  \psi_2^{\pm}(r, \pm \theta_1^*) &= \psi_1^\pm (r, \pm \theta_1^*)  
\end{aligned}\right.
\end{equation} 
on $\mathcal{S}_2$ where $h_2^+, \, h_2^->0$.  Note that we may solve
\eqref{eqn:psi_2} explicitly using the method of characteristics to
obtain the following expression for $\psi_2^\pm$:
\begin{align}
 \psi_2^\pm (r,\theta)&= d_{12}^{\pm}\frac{r^p}{|\theta|^{\frac{p}{n}}} + d_{22}^{\pm}
  \frac{r^p}{|\theta|^q}
\end{align}
where
\begin{align*}
  d_{12}^{\pm}=|\theta_1^*|^{\frac{p}{n}} \psi_1^\pm(1, \pm\theta_1^*) -
h_2^\pm \frac{|\theta_1^*|^{\frac{p}n-q}}{qn -p} \quad\quad\text{and}\quad\quad
d_{22}^{ \pm} = \frac{h_2^\pm}{qn-p}.
\end{align*}
Observe that $\psi_2^\pm$ consists of two terms, each of which is
homogeneous under $S_\alpha^\lambda$ for every $\alpha\geq 0$.
Moreover,
\begin{align}
\label{eqn:psi2ineq}
\psi_2^\pm (r,\theta) \geq |\theta_1^*|^{\frac{p}{n}} \psi_1^\pm(1,
\pm \theta_1^*) \frac{r^p}{|\theta|^{\frac{p}{n}}}
\end{align} 
on $\mathcal{S}_2$.  Thus we see that $\psi_2 >0$ on $\mathcal{S}_2$
and $\psi_2(r, \theta) \rightarrow \infty$ as $r\rightarrow \infty$,
$(r, \theta) \in \mathcal{S}_2$.

\subsubsection*{The inductive construction in the remaining transport regions}

We can now use the same idea employed above for $\psi_2$ to define
$\psi_3, \ldots, \psi_{j+3}$ inductively.  Thus for $m=3, \ldots,
j+3$, define $\psi_m$ as the solution of
\begin{equation}
\label{eqn:gen_trans}
\left\{\begin{aligned}
  (T_{m} \psi_{m}^{\pm})(r, \phi) &= -h_{m}^{\pm} r^{p_m}|\phi|^{-q_{m}} \\
  \psi_{m}^{\pm}(r, \pm \phi^*) &= \psi_{m-1}(r, \phi_{m-1}(\pm
  \phi^*))
\end{aligned}\right.
\end{equation}
for all $(r, \theta(r, \phi))\in \mathcal{S}_m$ where $p_m,
q_m, h_m^\pm>0$.  We again recall that we have suppressed the index
$m$ in $(r, \phi)$ using our convention.  We have also suppressed the
second index $m$ in $p_m$ and $q_m$ above; that is, $p_{m,m}=p_m$ and
$q_{m,m}=q_m$.    

Inductively, $p_{m}$ and $q_m$ are chosen to satisfy  
\begin{alignat}{2}
\notag &p_2 = p, \qquad && q_2 = q,\\
\label{eqn:rho_eta}& p_{m} = p_{m-1}+ q_{m-1}, \qquad && m=3, \ldots, j+3\\
\notag & q_{m} \in \Big(
q_{m-1}\vee\tfrac{p_m}{n+m-2},1\Big) \qquad && m=3, \ldots, j+3.
\end{alignat}
While these choices at the outset may seem mysterious, they are all
determined by the exit distribution of the diffusion generated by $A$
from $\mathcal{S}_{j+4}$ and by the scaling relationships of the
$\psi_i$'s along common boundaries. For further information, see the discussion in
Section~\ref{I-sec:scalingRelationDifRegion} of Part I \cite{HerzogMattingly2013I}.

We now prove a lemma which gives an expression for $\psi_m$ which is
convenient  for further analysis.  Although we will need them, at first
glance it is important to ignore the many relations that the constants
in the statement of the result satisfy.  The basic form of
$\psi_m^\pm$ is what is most important.
\begin{lemma}
\label{lem:indtrans}
For each $m=3, \ldots,j+3$ we may write 
\begin{align}
\label{eqn:gentrexp}
\psi_{m}^{\pm}(r, \phi) = \sum_{l=1}^{m} d_{l}^{\pm} \frac{r^{p_{l}}}{|\phi|^{q_{l}}}
\end{align} 
where the positive constants $p_{l}=p_{l,m}$ and $q_l=q_{l,m}$ satisfy
\begin{xalignat}{2}
\label{eqn:ind}
\notag &p_{1,2}=p_{2,2}= p \qquad \qquad && q_{1,2}= \frac{p}{n},\, \, q_{2,2}=q,\\
\notag &p_{l,m} = p_{l , m-1} + q_{l,m-1}&&l < m, \,\,  l,m = 3, \ldots, j+3\\
&p_{m,m} = p_{m-1, m}=p_m && m =3,\ldots, j+3 \\
\notag &q_{l,m}= \frac{p_{l,m}}{n+m-2},&& l<m, \, \, l,m  =3, \ldots, j+3\\
\notag & q_{m,m} = q_m, && m = 3,\ldots, j+3.
\end{xalignat}
Moreover, the constants $d_{l}^\pm =d_{l,m}^\pm$ are such that $\psi_m
>0$ on $\mathcal{S}_m$ and
\begin{align}
\label{eqn:blowup}
\psi_m(r, \phi) \rightarrow \infty
\end{align}
as $r\rightarrow \infty$, $(r, \theta(r,\phi))\in \mathcal{S}_m$.  
\end{lemma}

\begin{remark}
\label{rem:constord}
By inducting on $m$, notice that \eqref{eqn:ind} and \eqref{eqn:rho_eta} imply the
following ordering of the constants $p_{l,m}$, $q_{l,m}$ for $m>2$
\begin{align}
\label{eqn:constord}
&p_{1,m}< p_{2,m}< \cdots < p_{m-1,m}= p_{m,m}\\
\nonumber &q_{1,m}< q_{2,m}< \cdots < q_{m-1,m}< q_{m,m}<1.  
\end{align} 
The above relations will be especially helpful later when we do
asymptotic analysis of $\psi_m$.
\end{remark}

\begin{remark}
  In the proof of Lemma \ref{lem:indtrans} we will, in addition,
  derive some properties of the constants $d_l^\pm= d_{l,m}^{
    \pm}$. These will be collected in the statement of Corollary~\ref{cor:d_k} below, and they will be used, in particular, to show
  that we can choose the constants $h_m^{\pm}>0$ in a natural way so
  that the symmetry property \eqref{eqn:pensym} is satisfied.
\end{remark}

\begin{corollary}
\label{cor:d_k}
For $l<m$ with $l,m \in \{3,\ldots, j+3\} $, define the following constants  
\begin{align*}
  b_{l,m}^{\pm} &= \frac{|\phi^*|^{q_{l,m}}}{|c_{m-1} \mp \phi^*
    |^{q_{l,m-1}}}, \qquad e_m = |\phi^*|^{q_{m-1,m}-q_{m,m}}.
\end{align*}
Then for $l<m-1, \, \, l,m \in \{3,\ldots, j+3 \}$,  we have
\begin{align*}
d_{l,m}^{ \pm }= d_{l,m-1}^{\pm} b_{l,m}^{ \pm}
\end{align*}
and for $m = 3, \ldots, j+3$
\begin{align*}
d_{m,m}^{ \pm } & =  \frac{h_{m}^{\pm}}{q_m (n+ m-2) - p_m}\\
d_{m-1,m}^{ \pm} &= d_{m-1, m-1}^{\pm} b_{m-1,m}^{\pm} - d_{m,m}^{ \pm} e_m .  
\end{align*}
\end{corollary}

Before proving the lemma and corollary above, we state another lemma
which shows that, assuming the conclusions of Lemma \ref{lem:indtrans} and Corollary \ref{cor:d_k}, we can pick the constants $h_m^\pm$ in a reasonable
way so as to have \eqref{eqn:pensym}.
\begin{lemma}
\label{lem:h_choice}Fixing a constant $K_0>0$, for all $\epsilon >0$ there exists a
constant $K_1>0$ so that the following holds.
If $h_1^+, h_2^+, \ldots, h_{j+3}^+$ is a  collection of
positive parameters with $h_i^+ \leq K_0$ for all $i$ then for any
$\phi^* \geq K_1$ there exist a unique choice of positive $h_1^-, h_2^-, \ldots, h_{j+3}^-$ so that 
\begin{equation*}
\psi_{j+3}^+(r, -\phi)= \psi_{j+3}^-(r, \phi)
\end{equation*}
for all $(r, \phi)$ with $(r, \theta(r, \phi)) \in
\mathcal{S}_{j+3}$ and  
\begin{align*}
|h_{m}^+-h_m^-|\leq \epsilon 
\end{align*}     
for all $m=1,2, \ldots, j+3$. 
\end{lemma}  

\begin{remark}
\label{rem:hs}
Later, we will use the
  parameters $h_i^+$ to ensure that the fluxes across the boundaries
  between osculating regions where $\theta >0$ have the desired sign just as we did in
  Section~\ref{I-sub:bfc} of Part I \cite{HerzogMattingly2013I}. We
  will then need to choose the $h_i^-$ to both satisfy the boundary flux
  condition and make the $\psi_{j+3}$ have the desired symmetry. Note that since
  as $\phi^*\rightarrow \infty$ the regions become increasingly
  symmetric in the angular variable $\phi$, it is intuitively clear
  that the $h_i^-$ which produce a symmetric $\phi$ are close the
  $h_i^+$ which were already chosen. Hence, the $h_i^-$ which produce
  symmetry also satisfy the needed boundary flux condition.
\end{remark}

\begin{remark}
  Notice that the choice of $\phi^*$ determined by the lemma above is
  consistent with our process of picking parameters as outlined in
  Remark \ref{rem:parameterchoice}.
\end{remark}

We first give the  proof of Lemma \ref{lem:indtrans} and Corollary \ref{cor:d_k}
together and then prove Lemma \ref{lem:h_choice}
immediately afterwards.

\begin{proof}[Proof of Lemma \ref{lem:indtrans} and Corollary \ref{cor:d_k}]
  The proof will be done by induction on $m\geq 3$.  Suppose first
  that $m=3$.  Using the method of characteristics, one can easily
  derive the desired expression for $\psi_3$ and all claimed relations
  in Lemma \ref{lem:indtrans} and Corollary \ref{cor:d_k}.  To check \eqref{eqn:blowup} is valid for $\psi_3^\pm$, consider the
  dynamics along $T_3$
\begin{align*}
\dot{r}=r\qquad \text{and}\qquad
\dot{\phi} = (n+1) \phi, 
\end{align*}
and let $\t= \inf_{t>0}\{ |\phi_t|=\phi^*\}$.  Using the
inequality \eqref{eqn:psi2ineq}, notice for all $(r, \phi)$ such that
$(r, \theta(r, \phi))\in \mathcal{S}_3$ 
\begin{align*}
\psi_3^{\pm}(r, \phi) &= \psi_2^\pm\Big(
r_\t, \theta(r_\t,
\phi_\t)\Big) + h_3^\pm \int_0^\t
\frac{r_t^{p_3}}{|\phi_t|^{q_3}}\, dt \\ 
\nonumber &\geq   \psi_2^\pm\Big( r_\t,
\theta(r_\t, \phi_\t)\Big) = \psi_2^\pm(
r_\t, r_\t^{-1}(\phi^*-\frac{\gamma_1}{n+1}))\\ 
\nonumber & \geq c \frac{r^{p_{1,3}}}{|\phi|^{q_{1,3}}}
\end{align*}
for some constant $c>0$.  Hence we now see that $\psi_3 >0$ on
$\mathcal{S}_3$ and $\psi_3 \rightarrow \infty$ as $r\rightarrow
\infty$, $(r, \theta(r, \phi)) \in \mathcal{S}_3$.

Now assume all conclusions are valid for some $m-1 \geq 3$.  Using the
method of characteristics and the inductive hypothesis, we can obtain
the claimed expression for $\psi_m$ as well as all relationships
between constants in the statements of Lemma \ref{lem:indtrans} and
Corollary \ref{cor:d_k}.  To obtain \eqref{eqn:blowup}, we may assume
inductively that
\begin{align*}
\psi_{m-1}^\pm(r, \phi) \geq c \frac{r^{p_{1,m-1}}}{|\phi|^{q_{1, m-1}}}
\end{align*} 
for all $(r, \phi)$ with $(r, \theta(r, \phi)) \in
\mathcal{S}_{m-1}$ where $c>0$ is a constant (which is in general
different from the one used above).  As before, consider the dynamics
along $T_m$:
\begin{align*}
\dot{r}=r \qquad \text{and}\qquad\dot{\phi}= (n+m-2) \phi
\end{align*}   
and let, recycling notation, $\t=
\inf_{t>0}\{|\phi_t|=\phi^* \}$.  Then we similarly obtain
\begin{align*}
\psi_m^{\pm}(r, \phi) &= \psi_{m-1}^\pm\Big(
r_\t , \phi_{m-1}(r_\t,
\phi_\t)\Big) + h_m^\pm \int_0^\t
\frac{r_t^{p_m}}{|\phi_t|^{q_m}}\, dt \\ 
\nonumber &\geq   \psi_{m-1}^\pm\Big(
r_\t, \phi_{m-1}(r_\t,
\phi_\t)\Big) = \psi_{m-1}^\pm( r_\t,
r_\t^{-1}(\phi^*-c_{m-1}))\\ 
\nonumber & \geq c \frac{r^{p_{1,m}}}{|\phi|^{q_{1,m}}}
\end{align*} 
for some $c>0$ which is different from the $c$ used above.
This now finishes the proof of the result.
\end{proof}

\begin{proof}[Proof of Lemma \ref{lem:h_choice}]
Let $h_1^+, h_2^+, \ldots, h_{j+3}^+$ be a bounded collection of positive parameters and fix $\epsilon >0$.
  We will see that there is a unique choice of $h_1^-, h_2^-, \ldots,
  h_{j+3}^-$ which gives
\begin{align}
\label{rel:d_k}
d_{m, j+3}^{+}= d_{m, j+3}^{-} 
\end{align}    
for all $m=1, \ldots, j+3$.  By Corollary~\ref{cor:d_k} and
Lemma~\ref{lem:indtrans}, the first conclusion of the lemma will then
follow immediately since the $\psi_{j+3}^\pm$ are a linear combination
functions with coefficients $d_{m, j+3}^\pm$ respectively.  The closeness of the $h$'s
will follow for all $\phi^*$ large enough by inspection of the choice of the
$h_j^{-}$'s giving \eqref{rel:d_k} for all $m=1, \ldots, j+3$.

We proceed inductively and begin by analyzing the equality
\begin{align*}
d_{j+3-m, j+3}^{+}= d_{j+3-m, j+3}^{-}
\end{align*}  
for $m=0$.  Note that Corollary \ref{cor:d_k} implies that 
\begin{align*}
d_{j+3, j+3}^{+}= d_{j+3, j+3}^{ -} \iff h_{j+3}^+=h_{j+3}^-.
\end{align*}  
This, in particular, forces us to choose $h_{j+3}^-=h_{j+3}^+$.  Now consider the equality 
$d_{j+2 ,j+3}^{+} = d_{j+2, j+3}^{-}$.   
By Corollary \ref{cor:d_k} again and the fact that $d_{j+3, j+3}^{+}= d_{j+3, j+3}^{-}$, notice
\begin{align*}
d_{j+2, j+3}^{ \pm} = d_{j+2, j+2}^{\pm}b_{j+2, j+3}^{ \pm} - d_{j+3, j+3}^{+} e_{j+3}.
\end{align*}
Hence 
\begin{align*}
  d_{j+2, j+3}^{+}=d_{j+2, j+3}^{-} \iff h_{j+2}^{-} = h_{j+2}^{ +}
  \frac{b_{j+2, j+3}^{ +}}{b_{j+2, j+3}^{ -}}
\end{align*}
implying that we must pick $$h_{j+2}^{-} = h_{j+2}^{ +}
\frac{b_{j+2,j+3}^{ +}}{b_{j+2,j+3}^{ -}}.$$ Using the expressions given in Corollary \ref{cor:d_k} for the $b$'s, a simple argument
employing Taylor's theorem gives the following asymptotic formula as
$\phi^*\rightarrow \infty$:
\begin{align*}
\frac{b_{j+2, j+3}^{+}}{b_{j+2, j+3}^{-}} = 1 + O\big((\phi^*)^{-1}\big). 
\end{align*}  
Therefore   
\begin{align}
\label{eqn:h_{j+2}}
h_{j+2}^{-}= h_{j+2}^{+}+ O\big( (\phi^*)^{-1}\big)
\end{align} 
as $\phi^* \rightarrow \infty$.  In particular, this implies that the
unique choice of $h_{j+2}^+$ (which is positive for $\phi^*$ large
enough) determined by the relation $d_{j+2, j+3}^{+}= d_{j+2,
  j+3}^{-}$ has the desired closeness property $|h_{j+2}^+-h_{j+2}^{-}| < \epsilon$ for all $\phi^*$ large enough.  To
continue by induction, we need one more step to see how to proceed in
general.  Notice that this is only necessary if $j\geq 1$ where
$n=2j+1$ or $n=2j+2$.  By Corollary \ref{cor:d_k}, observe that
\begin{align*}
  d_{j+1, j+3}^{\pm} = d_{j+1, j+2}^{ \pm }b_{j+1, j+3}^{ \pm} =
  (d_{j+1, j+1}^{\pm} b_{j+1, j+2}^{ \pm}- d_{j+2, j+2}^{ \pm}
  e_{j+2})b_{j+1, j+3}^{ \pm}
\end{align*}       
and, by the right most equality, $d_{j+1, j+3}^{+}= d_{j+1, j+3}^{ -}$ is equivalent to 
\begin{align*}
  d_{j+1, j+1}^{-} = \Big(d_{j+1, j+1}^{+} \frac{b_{j+1,
      j+2}^{+}}{b_{j+1, j+2}^{-}} - d_{j+2, j+2}^{ +}\frac{
    e_{j+2}}{b_{j+1, j+2}^{-}} \Big) \frac{b_{j+1, j+3}^{ +}}{b_{j+1,
      j+3}^{ -}} + d_{j+2, j+2}^{ -} \frac{e_{j+2}}{b_{j+1, j+2}^{
      -}}.
\end{align*} 
By \eqref{eqn:h_{j+2}} and Corollary \ref{cor:d_k}, we have
\begin{align*}
d_{j+2, j+2}^{+} &= d_{j+2, j+2}^{-} + O\big( (\phi^*)^{-1}\big)\\
\frac{e_{j+2}}{b_{j+1, j+2}^{ -}} &= (\phi^*)^{q_{j+1}-q_{j+2}}\Big( 1+  O\big( (\phi^*)^{-1}\big)\Big)
\end{align*}
as $\phi^*\rightarrow \infty$.  Again, by Taylor's theorem we also have
\begin{align*}
  \frac{b_{j+1, j+2}^{+}}{b_{j+1,j+2}^{-}} &= 1+ O\big( (\phi^*)^{-1}\big) \\
  \frac{b_{j+1, j+3}^{+}}{b_{j+1, j+3}^{-}} &= 1+ O\big(
  (\phi^*)^{-1}\big)
\end{align*} 
as $\phi^* \rightarrow \infty$.  Putting these formulas together,
since $h_1^+, h_2^+, \ldots, h_{j+3}^+$ were assumed to be bounded and $q_{j+1}-q_{j+2}<0$ by \eqref{eqn:rho_eta} we obtain
\begin{align*}
d_{j+1, j+1}^{ -} = d_{j+1, j+1}^{+} +O\big((\phi^*)^{-1}\big)\\
 h_{j+1}^{ -} = h_{j+1}^{ +} +O\big((\phi^*)^{-1}\big)
\end{align*}
as $\phi^* \rightarrow \infty$.  This finishes the result in this
case.  To see in general when $d_{m, j+3}^{ +}=d_{m, j+3}^{-}$ for
general $m=2,\ldots, j$, assume by induction that
\begin{align*}
d_{m+1, m+1}^{ -}= d_{m+1, m+1}^{+} + O\Big((\phi^*)^{-1}\Big)  
\end{align*}
and note by successively applying $d^\pm_{l,m}=d_{l,m-1}^\pm b_{l,m}^\pm$ we obtain
\begin{align*}
d_{m,j+3}^{ \pm} &= d_{m, j+2}^{ \pm}b_{m,j+3}^{ \pm}= d_{m,m+1}^{\pm} b_{m,m+2}^{ \pm} b_{m,m+3}^{ \pm}\cdots b_{m,j+3}^{ \pm}\\
&=(d_{m,m}^{ \pm} b_{m,m+1}^{ \pm} - d_{m+1,m+1}^{ \pm} e_{m+1})b_{m,m+2}^{ \pm} b_{m,m+3}^{ \pm}\cdots b_{m,j+3}^{ \pm}.  
\end{align*}
Therefore, $d_{m,j+3}^{ +}=d_{m,j+3}^{ -}$ is equivalent to
\begin{align*}
  d_{m,m}^{-} = \Big(d_{m,m}^{+} \frac{b_{m,m+1}^{+}}{b_{m,m+1}^{-}} -
  d_{m+1,m+1}^{ +}\frac{ e_{m+1}}{b_{m,m+1}^{ -}} \Big)
  \frac{b_{m,m+2}^{ +}\cdots b_{m,j+3}^{+}}{b_{m,m+2}^{ -} \cdots
    b_{m,j+3}^{ -}} + d_{m+1,m+1}^{ -} \frac{e_{m+1}}{b_{m,m+1}^{-}}.
\end{align*} 
Similarly, using Taylor's theorem and the asymptotic formulas above, we see that 
\begin{align*}
d_{m,m}^{ -} = d_{m,m}^{ +} +O\big((\phi^*)^{-1}\big)\\
 h_{m}^{ -} = h_{m}^{ +} +O\big((\phi^*)^{-1}\big).
\end{align*}
Thus we have established the result for $h_2^\pm, h_3^\pm, \ldots,
h_{j+3}^\pm$.  Finally, to obtain the equality
\begin{align*}
d_{1,j+3}^{ +} = d_{1, j+3}^{-}
\end{align*}  
realize that it is equivalent to the relation 
\begin{align}
\label{eqn:h_1c}
d_{1,2}^{ -} = d_{1,2}^{+} \frac{b_{1,3}^{ +} \cdots b_{1,j+3}^{ +}}{b_{1,3}^{-} \cdots b_{1,j+3}^{-}}.
\end{align}
Since 
\begin{align*}
d_{1,2}^{ \pm} = |\theta_1^*|^{\frac{p}{n}}\psi_{1}^\pm (1, \pm \theta_1^*) - h_2^\pm \frac{|\theta_1^*|^{\frac{p}{n}-q}}{qn-p}
\end{align*}
and
\begin{align*}
 \frac{b_{1,3}^{ -} \cdots b_{1,j+3}^{ -}}{b_{1,3}^{+} \cdots b_{1,j+3}^{+}} =1 + O\big( (\phi^*)^{-1}\big)
\end{align*}
as $\phi^*\rightarrow \infty$, one can easily deduce from
\eqref{eqn:psi_1c} that for fixed
$\theta_1^*$, as $\phi^*\rightarrow \infty$ the choice of $h_1^-$
determined by the symmetry condition  \eqref{eqn:h_1c} approaches $h_1^+$.  Note that this
finishes the proof of the result.
\end{proof}
Now that we have the desired symmetry we turn to defining the final
function $\psi_{j+4}$ in the region $\mathcal{S}_{j+4}$ where noise
does play a role.

\subsection{The construction in the noise region $\mathcal S_{j+4}$}

Here, let $(r, \phi)=(r, \phi_{j+3})$ and define $\psi_{j+4}$ on
$\mathcal{S}_{j+4}$ as the solution of the following PDE
\begin{align}
\begin{cases}
  A\psi_{j+4}(r, \phi) = -h_{j+4}\,  r^{p_{j+4}}\\
  \psi_{j+3} = \psi_{j+3} \text{ on } \partial( \mathcal{S}_{j+3}\cap
  \mathcal{S}_{j+4})
\end{cases}
 \end{align}
 for all $(r, \phi)$ such that $(r, \theta(r, \phi)) \in
 \mathcal{S}_{j+4}$ where $h_{j+4}>0$ and
\begin{align}
\label{eqn:pj4}
p_{j+4} = \begin{cases}
p_{j+3} + \frac{q_{j+3}}{2} & \text{ if }\, n=2j+1, \, j\geq 0\\
p_{j+3}+ q_{j+3} & \text{ if } \, n =2j+2, \, j \geq 0.  
\end{cases}
\end{align}
We assume that the reader is familiar with the content of Section~\ref{I-sec:constructionnoise} of \cite{HerzogMattingly2013I} which outlines how one is able to solve the PDE above.

To solve for $\psi_{j+4}$, for simplicity let 
\begin{align*}
p_{m,j+4} = \begin{cases}
p_{m,j+3} + \frac{q_{m,j+3}}{2} & \text{ if }\, n=2j+1, \, j\geq 0\\
p_{m, j+3}+ q_{m, j+3} & \text{ if } \, n =2j+2, \, j \geq 0  
\end{cases}
\end{align*}
for $m=1,2, \ldots, j+3$.  Also, let $(r_t, \phi_t)$ denote the
diffusion defined by $A$ and ${\tau= \inf_{t>0}\{(r_t, \phi_t)
\notin \mathcal{S}_{j+4}\}.}$ Recalling the definition of $\partial(\mathcal{S}_{j+3} \cap \mathcal{S}_{j+4})$, we then see that
\begin{align}
\label{eqn:psifin1}
\psi_{j+4}(r, \phi) &= \E_{(r, \phi)}\psi_{j+3}(r_\tau, \phi_\tau)+ h_{j+1}\E_{(r, \phi)}\int_0^{\tau}
r_t^{p_{j+4}} \, dt\\ 
\nonumber &=   \sum_{m=1}^{j+3} \frac{d_{m,j+3}^{+}}{(\eta^*)^{q_{m}}}
r^{p_{m}} \E_{(r, \phi)} e^{p_{m}\tau }+ \frac{h_{j+4}}{p_{j+4}}r^{p_{j+4}}\E_{(r, \phi)}( e^{p_{j+4}
  \tau}-1)
\end{align}
where we have concatenated $p_{m,j+4}$ and $q_{m, j+3}$ to $p_{m}$ and $q_m$ respectively.  

To see that the maps $(r, \phi) \mapsto \E_{(r, \phi)} e^{p_m
  \tau} $ for $m=1,2,\ldots, j+4$ are well-defined and smooth
on $\mathcal{S}_{j+4}$, first observe that the process
\begin{align*}
\eta_t = \begin{cases}
r_t^{\frac{1}{2}}\phi_t & \text{ if } n=2j+1\\
r_t \phi_t + c_{j+3} & \text{ if } n =2j+2
\end{cases}
\end{align*}
satisfies the Gaussian SDE
\begin{align}
\label{eq:eta}
d\eta_t = \Big( \tfrac{3}{2}n+1\Big)\eta_t\, dt + \sigma \, dW_t .    
\end{align} 
Hence, we may write  
\begin{align*}
  \tau = \begin{cases}
    \inf\{t>0 \, : \, \eta_t \notin [-\eta^*, \eta^*]  \} & \text{ if } n=2j+1\\
    \inf\{ t>0 \,: \, \eta_t \notin [-\eta^* + c_{j+2}, \eta^* +
    c_{j+2}]\} & \text{ if } n= 2j+2.
\end{cases}  
\end{align*}
Applying Lemma~\ref{I-lem:exp_webf} of Part I \cite{HerzogMattingly2013I}, by choosing
$\eta^*>|c_{j+2}|$ large enough, it suffices to show that the
constants $p_{m}=p_{m, j+4}$ satisfy
\begin{align*}
p_{1,j+4} < p_{2, j+4} < \cdots < p_{j+4, j+4} < \tfrac{3n+2}{2}.
\end{align*}  
The fact that 
\begin{align*}
p_{1,j+4} < p_{2, j+4} < \cdots < p_{j+4, j+4}
\end{align*}
follows by Remark \ref{rem:constord} and the definition of the
constants $p_{m, j+4}$, $m=1,2, \ldots, j+4$.  The remaining bound can
be obtained inductively in either case $(n=2j+1 \text{ or } n=2j+2)$
by using the definition of $p_{j+4, j+4}=p_{j+4}$, the relations
\eqref{eqn:rho_eta}, and the choice of $p\in
(0,n)$.

To show that $\psi_{j+4}$ is strictly positive on $\mathcal{S}_{j+4}$
and $\psi_{j+4}(r, \phi) \rightarrow \infty$ as $r\rightarrow \infty$
with $(r, \phi)$ such that $(r, \theta(r, \phi)) \in
\mathcal{S}_{j+4}$, using \eqref{eqn:psifin1} we see that for some
constant $c>0$
\begin{align*}
  \psi_{j+4}(r, \phi) &\geq \E_{(r, \phi)} \psi_{j+3}(r_{\tau}, \phi_{\tau})\\
  &\geq c \E_{(r, \phi)} \frac{r^{p_{1,
        j+3}}_\tau}{|\phi_{\tau}|^{q_{1, j+3}}}
\end{align*}
where the last inequality follows by the inductive argument proving
both Lemma \ref{lem:indtrans} and Corollary \ref{cor:d_k}.  We thus
obtain the desired bound
\begin{align*}
  \psi_{j+4}(r, \phi) &\geq c \E_{(r, \phi)} \frac{r^{p_{1, j+3}}_\tau}{|\phi_\tau|^{q_{1, j+3}}}\\
  &\geq \frac{c}{(\eta^*)^{q_{1, j+4}}} r^{p_{1, j+4}}\E_{(r, \phi)}
  e^{p_{1, j+4} \tau} > c' r^{p_{1, j+4}}
\end{align*}
for some $c'>0$.  

\subsection{Summary of the construction}

Now that we have finished defining our Lyapunov function on each region $\mathcal{S}_i$, $i =1, \ldots, j+4$, we pause for a moment to provide a summary of the construction up to this point.  In the following sections, we will finish proving Theorem~\ref{thm:lyapfunmadness} by making sure the boundary-flux conditions of Corollary~\ref{I-cor:fluxIto} of Part I \cite{HerzogMattingly2013I} are satisfied and that each $\psi_i$ is indeed a local Lyapunov function on its domain of definition $\mathcal{S}_i$.          

\subsubsection{Regions and Asymptotic Operators}

Recalling that $n=2j+1$ or $n=2j+2$, the analysis of Section~\ref{sec:PiecewiseBeh} uncovered the asymptotic operators $$T_1, \ldots, T_{j+3}, A$$ and corresponding regions where we expect each to approximate well the time-changed Markov generator $L$ as $r\rightarrow \infty$.  The analysis of this section is summarized in the following three tables.

\begin{remark}
Recall that the constants $c_i$, $i=2, \ldots,j+2$, were defined inductively and depend on the Taylor expansion of the coefficients of $L$ at $\theta=0$.  Also recall that $\theta_0^* \in (\frac{\pi}{2n}, \frac{\pi}{n})$ is fixed and the constants $\theta_1^*$, $\phi^*$ and $\eta^*$ are chosen in the way outlined in Remark~\ref{rem:parameterchoice}.   
\end{remark}
   
\begin{center}
\resizebox{\linewidth}{!}{
\begin{tabular}{| c | c | c|} 
\hline Region $\mathcal{S}_i$, $i=0, \ldots, j+2$  & Asymptotic Operator & Coordinates  \\
\hline $\mathcal{S}_0=\{ r\geq r^*, \,
  \theta_0^* \leq |\theta| \leq \tfrac{\pi}{n} \}\cap \mathcal{R}$ & $T_1=r\cos(n\theta) \partial_r + \sin(n\theta) \partial_\theta$  & $r,\, \theta$ \\
\hline  $\mathcal{S}_1 =  \{r\geq r^*, \,
  0<\theta_1^*\leq |\theta| \leq \theta_0^*  \}\cap \mathcal{R}$ & $T_1 =r\cos(n\theta) \partial_r + \sin(n\theta) \partial_\theta$ & $r,\, \theta$ \\
\hline $\mathcal{S}_2=\{ r\geq r^*,\, |\phi_3|\geq \phi^*,\, |\theta| \leq \theta_1^*\}\cap \mathcal{R}$ & $T_2=r\partial_r + n \theta \partial_\theta$ &$ r,\, \theta$\\
\hline $\mathcal{S}_3 = \{r\geq r^*, |\phi_{4}|\geq \phi^*, \, |\phi_3|\leq \phi^* \}\cap \mathcal{R}$ & $T_3 = r \partial_{r}  + (n+1) \phi_{3} \partial_{\phi_{3}}$ & $r, \,\phi_3=r \theta+ c_2$ \\
\hline $\mathcal{S}_4 = \{ r\geq r^* , | \phi_{5}| \geq \phi^*, \, |\phi_4| \leq \phi^* \} \cap \mathcal{R} $ & $T_4 = r \partial_{r} +(n+2) \phi_{4} \partial_{\phi_4}$ & $r, \,\phi_4= r\phi_{3} + c_{3}$\\
\hline $\vdots $ & $\vdots$ & $\vdots$\\ 
\hline $\mathcal{S}_m = \{ r\geq r^* , | \phi_{m+1}| \geq \phi^*, \, |\phi_m| \leq \phi^* \} \cap \mathcal{R} $ & $T_m = r \partial_{r} +(n+m-2) \phi_{m} \partial_{\phi_m}$ & $r, \, \phi_m= r\phi_{m-1} + c_{m-1}$\\
\hline $\vdots $ & $\vdots$ & $\vdots$\\
\hline $\mathcal{S}_{j+2} = \{ r\geq r^* , | \phi_{j+3}| \geq \phi^*, \, |\phi_{j+2}| \leq \phi^* \} \cap \mathcal{R} $ & $T_{j+2} = r \partial_{r} +(n+j) \phi_{j+2} \partial_{\phi_{j+2}}$ & $r, \, \phi_{j+2}= r\phi_{j+1} + c_{j+1}$\\
\hline \end{tabular}}
\end{center}

\vspace{0.1in}

\begin{center}
\resizebox{\linewidth}{!}{
\begin{tabular}{|c|c|c|}
\hline Regions $S_{j+3}$, $\mathcal{S}_{j+4}$, $\textcolor{blue}{n=2j+1}$ & Asymptotic Operator  & Coordinates\\
\hline $\mathcal{S}_{j+3}= \{\,r\geq r^*, \,  \eta^*
r^{-\frac{1}{2}}\leq |\phi_{j+3}| \leq \phi^*  \} \cap \mathcal{R} $
 & $T_{j+3} = r \partial_{r} +(3j+2) \phi_{j+3} \partial_{\phi_{j+3}}$ & $r, \, \phi_{j+3}= r\phi_{j+2} + c_{j+2}$\\
\hline $\mathcal{S}_{j+4}= \{ r\geq r^*, \,  |\phi_{j+3}|\leq \eta^* r^{-1/2} \}\cap \mathcal{R}$ & $ A = r \partial_{r} +(3j+2)
  \phi_{j+3} \partial_{\phi_{j+3}} + \frac{\sigma^2}{2
    r} \partial_{\phi_{j+3}}^2$ &$r, \, \phi_{j+3}= r\phi_{j+2} + c_{j+2}$  \\
\hline
\end{tabular}}
\end{center}

\vspace{0.1in}

\begin{center}
\resizebox{\linewidth}{!}{
\begin{tabular}{| c | c | c|} 
\hline Regions $\mathcal{S}_{j+3}$, $\mathcal{S}_{j+4}$, $\textcolor{blue}{n=2j+2}$  & Asymptotic Operator & Coordinates  \\
\hline $\mathcal{S}_{j+3}= \{\,r\geq r^*, \,  \eta^*
r^{-1}\leq |\phi_{j+3}| \leq \phi^*  \} \cap \mathcal{R} $
 & $T_{j+3} = r \partial_{r} +(3j+3) \phi_{j+3} \partial_{\phi_{j+3}}$ & $r, \, \phi_{j+3}= r\phi_{j+2} + c_{j+2}$\\
\hline $\mathcal{S}_{j+4}= \{ r\geq r^*, \,  |\phi_{j+3}|\leq \eta^* r^{-1} \}\cap \mathcal{R}$ & $ A = r \partial_{r} +[(3j+3)
  \phi_{j+3}  + \gamma_1^{(j+3)} r^{-1}] \partial_{\phi_{j+3}}+ \frac{\sigma^2}{2
    r} \partial_{\phi_{j+3}}^2$ &$r, \, \phi_{j+3}= r\phi_{j+2} + c_{j+2}$  \\
\hline
\end{tabular}}

\end{center}

\subsubsection{Properties of the Lyapunov Function in Each Region}

Below we give a summary of some of the basic properties of our Lyapunov function $\Psi$ on the principal wedge $\mathcal{R}$ in each region $\mathcal{S}_i$.  We recall that the constants $p$, $q$ satisfy $p\in (0, n)$, $q\in (\frac{p}{n},1)$ and the constants $d_{l,m}^\pm$ are determined by the boundary conditions in each Poisson equation.  As mentioned in Remark~\ref{rem:hs}, the constants $h_i^\pm >0$ will be chosen so that both the reflective symmetry \eqref{eqn:pensym} and the boundary-flux conditions are satisfied.       

\vspace{0.1in}
  
\begin{center}
\begin{tabular}{| c | c | c| c|} 
\hline Region $\mathcal{S}$ & Asymptotic Operator $\mathcal{O}$ & $\Psi\vert_{\mathcal{S}} $ & $\mathcal{O}(\Psi \vert_{\mathcal{S}})$ on $\mathcal{S}$ \\
\hline $\mathcal{S}_0$ & $T_1$ &$ r^p$ & $p\cos(n\theta) r^{p} $ \\
\hline $\mathcal{S}_1$ & $T_1$ & eqn. \eqref{eqn:psi_1c} & $-h_1^\pm r^p|\theta|^{-q}$ \\
\hline $\mathcal{S}_2 $ & $T_2$ & $d_{12}^\pm \frac{r^p}{|\theta|^{p/n}} + d_{22}^\pm \frac{r^p}{|\theta|^q}$ & $-h_2^\pm r^p |\theta|^{-q}$ \\
\hline $\mathcal{S}_3$ & $T_3 $ & $\sum_{l=1}^{3} d_{l,3}^\pm \frac{r^{p_{l,3}}}{|\phi|^{q_{l,3}}} $ & $-h_3^\pm r^{p_3} |\phi_3|^{-q_3}$ \\
\hline $\vdots $ & $\vdots $ & $\vdots $ & $\vdots$\\
\hline $\mathcal{S}_m$ & $T_m$ & $\sum_{l=1}^m d_{l,m}^\pm \frac{r^{p_{l,m}}}{|\phi_m|^{q_{l,m}}}$ & $-h_m^\pm r^{p_m} |\phi_m|^{-q_m}$\\
\hline $\vdots $ & $\vdots $ & $\vdots $ & $\vdots$\\
\hline $\mathcal{S}_{j+3}$ & $T_{j+3}$ & $\sum_{l=1}^{j+3} d_{l, j+3}^\pm \frac{r^{p_{l, j+3}}}{|\phi_{j+3}|^{q_{l, j+3}}}$ & $-h_{j+3}^\pm r^{p_{j+3}} |\phi_{j+3}|^{-q_{j+3}}$\\
\hline $\mathcal{S}_{j+4}$ & $A$ & eqn. \eqref{eqn:psifin1} & $-h_{j+4} r^{p_{j+4}}$\\
\hline 
\end{tabular}
\end{center}

\section{Boundary-flux calculations}
\label{sec:boundaryFluxCalculations}
Here show how one can choose the positive parameters $h_i^+$,
$\theta_1^*$, $\phi^*$, $\eta^*$ so that the jump conditions of
Corollary~\ref{I-cor:fluxIto} of Part I \cite{HerzogMattingly2013I} are
also satisfied.  We must be careful to see that all choices are
consistent with Remark \ref{rem:parameterchoice} and Lemma
\ref{lem:h_choice}.
Each boundary has two disjoint parts, implying that we must check two,
although very similar, flux conditions.  We proceed from boundary to
boundary, starting with the:

\subsection{Boundary between  $\mathcal{S}_0$ and $\mathcal{S}_1$}

We begin on the side of the boundary where $\theta>0$.  We must pick the parameters so that  
\begin{align}
\label{eqn:f1p}
\bigg[ \frac{\partial \psi_0}{\partial \theta } -\frac{\partial
  \psi_1^+}{\partial \theta}\bigg]_{\theta=\theta_0^*}\leq 0
\end{align}
for $r\geq r^*$.  By inspection of the formula \eqref{eqn:psi_1c}, we
first note that $\psi_1^+(r,\theta)= r^p \psi_1^+(1,\theta)$.  Using
this and the equation \eqref{caseII:psi_1} defining $\psi_1^+$, observe also that
\begin{align*}
  -h_1^+ r^p |\theta|^{-q} = \frac{\partial
    \psi_1^+}{\partial r}r \cos(n\theta)+ \frac{\partial
    \psi_1^+}{\partial \theta} \sin(n\theta).
\end{align*}
Rearranging this produces
\begin{align}\label{thetaDeriv0}
  \frac{\partial \psi_1^+}{\partial \theta}=- r^p \Big(\frac{p
    \cos(n \theta) \psi_1^+(1,\theta) + h_1^+ |\theta|^{-q}}{ \sin(n \theta)}\Big).
\end{align}
Therefore combining  $\frac{\partial \psi_0}{\partial \theta}=0$ with
\eqref{thetaDeriv0} gives
\begin{align*}
  \bigg[ \frac{\partial \psi_0}{\partial \theta } -\frac{\partial
    \psi_1^+}{\partial \theta}\bigg]_{\theta=\theta_0^*} = r^p
  \Big(\frac{p \cos(n \theta_0^*) \psi_1^+(1,\theta_0^*) + h_1^+
    |\theta_0^*|^{-q}}{ \sin(n \theta_0^*)}\Big).
\end{align*}
Because $\psi_1^+(1, \theta_0^*)=1$, $\sin(n\theta_0^*)>0$ and
$\cos(n\theta_0^*)<0$, picking
\begin{align}
\label{eqn:flux_h1m}
0<h_1^+ < p (\theta_0^*)^q |\cos(n\theta_0^*)| 
\end{align}
results in \eqref{eqn:f1p}.

On the side of the boundary where $\theta <0$, we must see that this
choice of $h_1^+$ also implies
\begin{align}
\label{eqn:f1m}
\bigg[ \frac{\partial \psi_1^-}{\partial \theta } -\frac{\partial
  \psi_0}{\partial \theta}\bigg]_{\theta=-\theta_0^*}\leq 0
\end{align}
for $r\geq r^*$.  By Lemma \ref{lem:h_choice}, we have already picked
$h_1^-$ and we recall that as $\phi^* \rightarrow \infty$,
$h_1^-\rightarrow h_1^+$.  Using the very same process as above,
\eqref{eqn:f1m} is satisfied provided
\begin{align}
0<h_1^- < p (\theta_0^*)^q |\cos(n\theta_0^*)|.
\end{align}  
Therefore, both quantities can be seen to be negative by first picking 
\begin{align}
\label{eqn:flux_h1p}
0<h_1^+ < p (\theta_0^*)^q |\cos(n\theta_0^*)| 
\end{align}
and then taking $\phi^*>0$ sufficiently large.  Note that this is
consistent with the flow of choices outlined in Remark
\ref{rem:parameterchoice}.


\subsection{Boundary between  $\mathcal{S}_1$ and $\mathcal{S}_2$}

We proceed in a similar fashion by first doing the computation on the
side of the boundary where $\theta>0$.  We first show that
\begin{align}
\label{eqn:f2p}
\bigg[ \frac{\partial \psi_1^+}{\partial \theta } -\frac{\partial
  \psi_2^+}{\partial \theta}\bigg]_{\theta=\theta_1^*}\leq 0
\end{align}    
for $r\geq r^*$ whenever $\theta_1^*>0$ is small enough.  Using
$\psi_2^+(r, \theta)=r^p \psi_2^+(1, \theta)$ and the equation
$\psi_2^+$ satisfies, we obtain
\begin{align*}
\frac{\partial \psi_2^+}{\partial \theta} = - r^p \bigg[ \frac{p\psi_2^+(1, \theta) + h_2^+ |\theta|^{-q}}{n \theta} \bigg].
\end{align*}
Since $\psi_1^+(1,\theta_1^*) = \psi_2^+(1, \theta_1^*)$, notice
\begin{align*}
  &\bigg[  \frac{\partial \psi_1^+}{\partial \theta}-\frac{\partial \psi_2^+}{\partial \theta }\bigg]_{\theta=\theta_1^*} \\
  &= -r^p \bigg[- \frac{p \psi_1^+(1, \theta_1^*) + h_2^+
    |\theta_1^*|^{-q}}{n \theta_1^*}+\frac{p
    \cos(n \theta_1^*) \psi_1^+(1,\theta_1^*) + h_1^+ |\theta_1^*|^{-q}}{ \sin(n \theta_1^*)}\bigg]\\
  &= -\frac{ r^p }{|\theta_1^*|^{q+1}} \bigg[ \bigg( \frac{p
    \cos(n\theta_1^*)}{\sin(n\theta_1^*)}- \frac{p}{n\theta_1^*}
  \bigg)\psi_1^+(1, \theta_1^*) |\theta_1^* |^{q+1} + \bigg(
  \frac{h_1^+}{\sin(n\theta_1^*)}-
  \frac{h_2^+}{n\theta_1^*}\bigg)|\theta_1^*|\bigg].
\end{align*}
The expression \eqref{eqn:psi_1c} implies that $\psi_1^+(1,
\theta_1^*) |\theta_1^*|^{q}\rightarrow 0$ as $\theta_1^* \downarrow
0$.  Using this fact and expanding $\sin(n\theta_1^*)$ and
$\cos(n\theta_1^*)$ in power series about $\theta_1^*=0$, we arrive at
the asymptotic formula
\begin{align*}
  &\bigg[ \bigg( \frac{p \cos(n\theta_1^*)}{\sin(n\theta_1^*)}- \frac{p}{n\theta_1^*} \bigg)\psi_1^+(1, \theta_1^*) |\theta_1^* |^{q+1} + \bigg( \frac{h_1^+}{\sin(n\theta_1^*)}- \frac{h_2^+}{n\theta_1^*}\bigg)|\theta_1^*|\bigg]\\
  &=\bigg( \frac{h_1^+}{n}- \frac{h_2^+}{n}\bigg)+ o(1)
\end{align*}
as $\theta_1^*\downarrow 0$.  Therefore, for every choice of 
\begin{align}
h_2^+ < h_1^+
\end{align} 
we may pick $\theta_1^*>0$ sufficiently small so that the flux across
the boundary where $\theta>0$ is negative.  On the side of the
boundary where $\theta<0$, a similar line of reasoning shows that the
choice
\begin{align}
h_2^-<h_1^-
\end{align}  
results in a negative flux for all $\theta_1^*>0$ small.  Recall,
also, that this is consistent with both Remark
\ref{rem:parameterchoice} and Lemma \ref{lem:h_choice} by, after
choosing $\theta_1^*>0$ small, choosing $\phi^*>0$ large.

\subsection{Boundary between $\mathcal{S}_2$ and $\mathcal{S}_3$}
For illustrative purposes, we perform one more boundary-flux estimate before proceeding on to the
general, inductive calculation in the remaining transport regions.  We
begin on the side of the boundary where $\phi_3>0$.  Note that for
$\phi^*>0$ large, it is also true that $\theta>0$ on this side.

As opposed to the previous cases, it is more convenient to use the
explicit expressions obtained for $\psi_2^\pm$ and $\psi_3^\pm$.  In
doing this, we first note that
\begin{align}
  \label{eqn:flux_23_1}\bigg[\frac{\partial \psi_2^+}{\partial
    \theta}- \frac{\partial \psi_3^+}{\partial \theta}
  \bigg]_{\phi_3=\phi^*}\leq C_1(\phi^*) r^{p+q+1} + C_2(\phi^*)
  r^{p+\frac{p}{n}+1}
\end{align}
where 
\begin{align*}
  C_1(\phi^*)=- \frac{ q_2 d_{2,2}^{+}}{|\phi^*- c_1 |^{q_2+1}} +
  \frac{q_{2,3} d_{2,3}^{+}}{|\phi^*|^{q_{2,3}+1}}+
  \frac{q_{3}d_{3,3}^{+}}{|\phi^*|^{q_{3}+1}}
\end{align*}
and $C_2(\phi^*)$ is a constant depending on $\phi^*$.  Our goal is to
see that for $\phi^*$ large enough, $C_1(\phi^*)<0$.  Hence for
$r^*>0$ large enough, the quantity \eqref{eqn:flux_23_1} will also be
negative.  Recalling the dependence of $d_{2,3}^+$ on $\phi^*$ in
Corollary \ref{cor:d_k} and that $q_3>q_2 $, we note that
\begin{align*}
C_1(\phi^*) &= - (\phi^*)^{-q_2-1}\big((q_2 - q_{2,3})d_{2,2}^{ +} +  o(1) \big) 
\end{align*}  
as $\phi^*\rightarrow \infty$. Since $d_{2,2}^{+}$ is positive  and independent of $\phi^*$ and  
\begin{align*}
q_{2,3}= \frac{p_{2,3}}{n+1}=\frac{p_{2,2}+q_{2,2}}{n+1} = \frac{p+q}{n+1}
\end{align*}
where $q_2 = q \in ( \frac{p}{n},1)$, we find that $q_2 > q_{2,3}$.
Hence, choosing $\phi^*>0$ large enough, $C_1(\phi^*)$ is negative.
Thus for $r^*>0$, the quantity on the left-hand side of
\eqref{eqn:flux_23_1} is also negative.  A nearly identical
computation will yield the desired result on the other side of the
boundary.

\subsection{The boundaries between the remaining transport regions}
We now consider the flux across the two boundaries between
$\mathcal{S}_m$ and $\mathcal{S}_{m+1}$ where $k =3,\ldots, j+2$.  As
done in the previous case, we focus on the side of the boundary where
$\phi_{m+1}>0$.  Note, too, with $\phi^*>0$ large enough, $\phi_{m}$
is also positive on that side of the boundary.  Using the expressions
derived in Lemma \ref{lem:indtrans}, realize that
\begin{align}
\label{eqn:psimbf}
&\bigg[\frac{\partial \psi_m^+}{\partial \theta} -\frac{\partial \psi_{m+1}^+}{\partial \theta}\bigg]_{\phi_{m+1}=\phi^*} \leq C_1(\phi^*) r^{p_{m+1} +m-1} + C_2(\phi^*) r^{c}, 
\end{align}
where $c< p_{m+1}+m-1$,
\begin{align*}
  C_1(\phi^*) = - q_{m} \frac{d_{m,m}^{+}}{(\phi^* -c_{m})^{q_m+1}} +
  q_{m,m+1} \frac{d_{m,m+1}^{+}}{(\phi^*)^{q_{m,m+1}+1}}+ q_{m+1}
  \frac{d_{m+1, m+1}^{+}}{(\phi^*)^{q_{m+1}+1}}
\end{align*}
and $C_2(\phi^*)$ is a constant that depends on $\phi^*$.  Using
Corollary \ref{cor:d_k} to write out $d_{m,m+1}^+$ and recalling that
$q_{m+1}>q_m$, note that as $\phi^*\rightarrow \infty$
\begin{align*}
  C_1(\phi^*) = -(\phi^*)^{-q_{m}-1}\bigg( (q_{m}- q_{m,m+1})
  d_{m,m}^+ + o(1) \bigg).
\end{align*} 
Recalling that 
\begin{align*}
q_{m,m+1}= \frac{p_{m,m+1}}{n+m-1}= \frac{p_m +q_m}{n+m-1}
\end{align*}
and $q_m \in ( \frac{p_m}{n+m-2}, 1)$ we see that $q_m > q_{m,m+1}$
giving that $C_1(\phi^*)$ is negative for $\phi^*$ large enough.  Thus
for $r^*>0$ large enough the quantity on the left-hand side of
\eqref{eqn:psimbf} is negative.  A nearly identical result holds on
the other side of the boundary.

\subsection{Boundary between $\mathcal{S}_{j+3}$ and $\mathcal{S}_{j+4}$}
\label{sec:webf}  
In the following computation, we will need to employ
Lemma~\ref{I-lem:exp_webf} of Part I \cite{HerzogMattingly2013I} since the
expression for $\psi_{j+4}$ in \eqref{eqn:psifin1} is not explicit.
Also, we only show the case when $n=2j+1$, $j\geq 0$, as the other
case is similar.

Consider the side of the boundary where $\phi_{j+3}>0$.  Recalling the
notation $G_{a,c}$ introduced in Section~\ref{I-sec:whatisG} of Part I
\cite{HerzogMattingly2013I}, observe that
\begin{align}
  \label{eqn:psifinbf} &\bigg[ \frac{\partial\psi_{j+3} }{\partial
    \theta} -\frac{\partial\psi_{j+4} }{\partial
    \theta}\bigg]_{\eta=\eta^*}\leq C_1( \eta^*) r^{p_{j+4} +
    j+\frac{3}{2}} + C_2(\eta^*) r^c
\end{align} 
for some $c<p_{j+4}+ j+\tfrac{3}{2}$ where
\begin{align*}
  C_1(\eta^*) &= -\bigg[\frac{d_{j+3,j+3}^+}{(\eta^*)^{q_{j+3}}} +
  \frac{h_{j+4}}{p_{j+4}}\bigg] G_{p_{j+4},0}'(\eta^*)-\frac{d_{j+3,
      j+3}^+\, q_{j+3}}{(\eta^*)^{q_{j+3}+1}}
\end{align*} 
and $C_2(\eta^*)$ is a constant which depends on $\eta^*$.  Choosing 
\begin{align*}
h_{j+4}= h p_{j+4} (\eta^*)^{-q_{j+3}} 
\end{align*}
for some $h>0$ and applying the Lemma~\ref{I-lem:exp_webf} of
Part I \cite{HerzogMattingly2013I}, realize that as $\eta^*\rightarrow \infty$
\begin{align*}
C_{1}(\eta^*) = (\eta^*)^{-q_{j+3}-1} \bigg(d_{j+3,j+3}^+\frac{2 p_{j+4} }{3n+2} - d_{j+3, j+3}^+ q_{j+3}+ \frac{2 h}{3n+2} + o(1)\bigg).
\end{align*}    
Using \eqref{eqn:pj4} and the relations $q_{j+3}> p_{j+3}/(n+j+1)$ and $n=2j+1$, we see that 
\begin{align*}
\frac{2p_{j+4}}{3n+2}<q_{j+3}
\end{align*}
Picking $h$ small enough implies that $C_1(\eta^*)<0$ for $\eta^*>0$
large enough.  Therefore choosing $r^*>0$ large enough implies that
the quantity on the left-hand side of \eqref{eqn:psifinbf} is
negative.  A similar result is easily seen to hold on the other side
of the boundary.

\section{Checking the Global Lyapunov Bounds}
\label{sec:CheckingTheGlobalLyapunovBounds}
\subsection{Checking the Local Lyapunov Property}
\label{LLPcheck}

Here we check that the approximating operators $T_1, T_2, \ldots,
T_{j+3}, A$ were chosen correctly so that $\psi_0, \psi_1, \ldots,
\psi_{j+4}$ are actually locally Lyapunov functions on their
respective domains $\mathcal{S}_0, \mathcal{S}_1, \ldots,
\mathcal{S}_{j+4}$.  This simply involves replacing
each asymptotic operator with $L$ and estimating the remainder locally on each region.  Factoring in the time change, the required bound for $\mathcal{L} \psi_i$ on $\mathcal{S}_i$ will then follow easily.  

\subsubsection*{Region $\mathcal{S}_0$}

Since $\psi_0(r, \theta)=r^p$, it is not hard to see that as
$r\rightarrow \infty$, $(r, \theta) \in \mathcal{S}_0$,
\begin{align}
\label{eqn:llpsi_1as}
L \psi_0 (r, \theta) &= p r^{p} \cos(n\theta) + o(r^{p}).  
\end{align}
Since $\cos(n\theta) \leq -c <0$ for $(r, \theta) \in \mathcal{S}_0$,
the relation \eqref{eqn:llpsi_1as} implies that there exist positive
constants $c_0, d_0$ such that
\begin{align*}
L\psi_0(r, \theta) &\leq - c_0 r^{p} + d_0
\end{align*}   
for all $(r, \theta) \in \mathcal{S}_0$.  Undoing the time change, we
see that there exist positive constants $C_0, D_0$ such that on
$\mathcal{S}_0$
\begin{align}
\label{eqn:psi_0b1}
\mathcal{L} \psi_0 (r, \theta)\leq -C_0 r^{p+n} + D_0.  
\end{align}

\subsubsection*{Region $\mathcal{S}_1$}  
First observe that by definition of $\psi_1^\pm$, we see that  
\begin{align*}
L\psi_1^\pm (r, \theta)&= T_1 \psi_1^\pm (r, \theta) + (L-T_1) \psi_1^\pm (r, \theta)\\
&= - h_1^\pm \frac{r^p}{|\theta|^q} + (L-T_1) \psi_1^\pm (r, \theta)  
\end{align*}
on $\mathcal{S}_1$ where the $\pm$ indicates the values of the
functions above when $\theta$ is, respectively, positive or negative.
To bound the remainder term $(L-T_1) \psi_1^\pm(r, \theta)$, recall by
\eqref{eqn:psi_1c} we may write
$\psi_1^\pm(r, \theta) = r^p \psi_1^\pm(1,\theta)$ where the mapping
$\theta \mapsto \psi_{1}^\pm(1, \theta)$ is a smooth and positive
function in $\theta$ for all $0<\theta_1^*\leq |\theta|\leq
\theta_0^*$.  In particular, since $0<\theta_1^*\leq |\theta|\leq
\theta_0^*$ for $(r, \theta) \in \mathcal{S}_1$, we see that as
$r\rightarrow \infty$, $(r, \theta) \in \mathcal{S}_1$,
\begin{align*}
L \psi_1^\pm (r, \theta)&=
- h_1^\pm \frac{r^p}{|\theta|^q} + o(r^p).  
\end{align*}
From this, we obtain the inequality
\begin{align*}
L \psi_1^\pm (r, \theta)&\leq 
- c_1 \frac{r^p}{|\theta|^q} + d_1  
\end{align*}
for some constants $c_1, d_1>0$, for all $(r, \theta) \in
\mathcal{S}_1$.  Undoing the time change, we see that there exist
constants $C_1, D_1>0$ such that on $\mathcal{S}_1$
\begin{align}
\label{eqn:psi_1b1II}
\mathcal{L} \psi_1^\pm (r, \theta)&\leq 
- C_1 \frac{r^{p+n}}{|\theta|^q} + D_1 .   
\end{align}

\subsubsection*{Region $\mathcal{S}_2$}

By definition of $\psi_2^\pm$, first observe that on $\mathcal{S}_2$
\begin{align*}
  L \psi_2^\pm (r, \theta) &= T_2 \psi_2^\pm (r, \theta) +
  (T_1-T_2)\psi_2^\pm(r, \theta) + (L-T_1) \psi_2^\pm (r, \theta)\\ 
  &= -h_2^\pm \frac{r^p}{|\theta|^q} +(T_1-T_2)\psi_2^\pm (r, \theta)
  + (L-T_1) \psi_2^\pm (r, \theta) .
\end{align*}
Using the Taylor expansions for $\sin(n\theta)$ and $\cos(n\theta)$
notice that there exists a constant $C>0$ so that
\begin{align*}
&(T_1-T_2)\psi_2^\pm (r, \theta)\\
&\leq C \theta^2\bigg[\bigg(
  |\theta_1^*|^{\frac{p}{n}} \psi_1^\pm (1,\pm\theta_1^*) + h_2^\pm
  \frac{|\theta_1^*|^{\frac{p}n-q}}{qn -p}
  \bigg) \frac{r^p}{|\theta|^{p/n}} +  \frac{h_2}{qn-p} \frac{r^p}{|\theta|^q}\bigg]\\
  &\leq  C (\theta_1^*)^2\bigg[\bigg(
  |\theta_1^*|^{\frac{p}{n}} \psi_1^\pm(1,\pm\theta_1^*) + h_2^\pm
  \frac{|\theta_1^*|^{\frac{p}n-q}}{qn -p}
  \bigg) \frac{r^p}{|\theta|^{p/n}} +  \frac{h_2^\pm}{qn-p} \frac{r^p}{|\theta|^q}\bigg] 
\end{align*}     
for all $(r, \theta) \in \mathcal{S}_2$.  Since $\psi_1^\pm(1,
\pm\theta_1^*)=O((\theta_1^*)^{-1})$ as $\theta_1^* \downarrow 0$, it follows that for all $\theta_1^*>0$ sufficiently small
\begin{align*}
(T_1-T_2)\psi_2^\pm (r, \theta)
&\leq \frac{h_2^+ \wedge h_2^-}{2} \frac{r^p}{|\theta|^q} 
\end{align*} 
for all $(r, \theta) \in \mathcal{S}_2$.  Therefore, for all  $\theta_1^*>0$ small enough we have the bound    
\begin{align*}
L \psi_2^\pm (r, \theta) &\leq 
-\frac{h_2^\pm}{2} \frac{r^p}{|\theta|^q} + (L-T_1) \psi_2^\pm(r, \theta)
\end{align*} 
on $\mathcal{S}_2$.

To control the remaining term, first recall the
definition of the region $\mathcal{S}_2$.  Notice then that there
exists a positive constant $C=C(\phi^*, r^*)$ such that on
$\mathcal{S}_2$
\begin{align*}
(L-T_1) \psi_2^\pm(r, \theta) &\leq C(r^*, \phi^*) \frac{r^p}{|\theta|^q}
\end{align*}  
where $C(r^*, \phi^*)>0$ satisfies the following property:  For every $\epsilon >0$, there exists $K>0$ such that for $\phi^* \wedge r^* \geq K$
\begin{align*}
C(r^*, \phi^*) \leq \epsilon.
\end{align*}
Hence we may pick $K>0$ large enough so that for $\phi^*\wedge r^* \geq K$
\begin{align*}
L \psi_2^\pm(r, \theta) & \leq - c_2 \frac{r^p}{|\theta|^q} + d_2
\end{align*}  
for all $(r, \theta) \in \mathcal{S}_2$.  Undoing the time change, we then determine the existence of positive constants $C_2, D_2$ such that on $\mathcal{S}_2$
\begin{align}
\label{eqn:psi_2b1}
\mathcal{L} \psi_2^\pm \leq -C_2 \frac{r^{p+n}}{|\theta|^q} + D_2
\end{align}  

\begin{remark}
\label{rem:dasym}
Before proceeding onto the remaining regions, it is important to note that Corollary \ref{cor:d_k} and the relations \eqref{eqn:constord} imply that for $m \in \{ 3,4, \ldots, j+3\}$ and $l\in \{1,2, \ldots, m-1\}$:
\begin{align}
\label{eqn:asymphi}
d_{l,m}^\pm =O( (\phi^*)^{q_{l,m}-q_l}) \, \text{ as } \,\phi^* \rightarrow \infty  
\end{align}
where the constant in the asymptotic formula above is independent of $\theta_1^*$, $\eta^*$ and $r^*$.  The above fact will be helpful when controlling remainder terms in what follows.
\end{remark}

\subsubsection*{Region $\mathcal{S}_m$, $m=3, \ldots, j+2$}
In the following computations, it is helpful to consult
\eqref{eqn:asyLm} and the remainder estimate immediately below it.
For lack of better notation, we will also use $\psi^\pm_m$ to denote
the function of $(r, \theta)$ determined by $\psi^\pm_m =\psi_m^\pm(r,
\phi)$.  Let
\begin{align}
\label{eqn:N2}
N = \bigg(\frac{\sigma^2}{2 r^n} \partial_r^2\bigg)_{(r, \phi)}
\end{align}
and write 
\begin{align*}
L\psi^\pm_m(r, \theta) & = L_{(r, \phi)} \psi_{m}^\pm(r, \phi)\\
&= T_m \psi_{m}^\pm(r, \phi) + (L_{(r, \phi)} - T_m -N) \psi_{m}^\pm(r, \phi) + N \psi_{m}^\pm(r, \phi)\\
&= - h_m^\pm \frac{r^{p_m}}{|\phi|^{q_m}} + (L_{(r, \phi)} - T_m-N ) \psi_{m}^\pm(r, \phi) + N \psi_{m}^\pm(r, \phi)
\end{align*}
where each equality above is valid on $\mathcal{S}_m$.  We first focus
on estimating $(L_{(r, \phi)}- T_m-N) \psi_m^\pm(r, \phi)$ for all
$(r, \phi)$ such that $(r, \theta(r, \phi)) \in
\mathcal{S}_m$.  Using the simple nature of the expression derived for
$\psi_m^\pm$ as well as Remark \ref{rem:dasym}, we note that the bound
\begin{align*}
|(L_{(s, \phi)}- T_m-N) \psi_{m}^\pm(r, \phi)| \leq C_1(r^*, \phi^*) \frac{r^{p_m}}{|\phi|^{q_m}} 
\end{align*}      
holds on $\mathcal{S}_m$ where $C_1(r^*, \phi^*)$ is a constant which can be chosen to be a small as we wish by first picking $\phi^*>0$ large and then picking $r^*>0$ large.  Therefore, making such choices we see that 
\begin{align*}
L\psi^\pm_m(r, \theta) \leq - \frac{h_m^\pm}{2}\frac{r^{p_m}}{|\phi|^{q_m}} + N \psi_{m}^\pm(r, \phi).
\end{align*}
To estimate the remaining term $N \psi_{m}^\pm(r, \phi)$, first recall that 
\begin{align*}
& \phi = r^{m-2} \theta + r^{m-3} c_2+ \cdots + c_{m-1}
\end{align*}
and so we may write
\begin{align*}
N= \frac{\sigma^2}{2 r^n} \bigg(\partial_r + [(m-2)  r^{-1} \phi + r^{-1}Y(r)] \partial_\phi \bigg)^2
\end{align*}
where $Y$ is a polynomial in $r$ of degree at most $m-3$.  Using
this allows us to obtain a similar bound
\begin{align*}
|N\psi_m^\pm(s, \phi)| \leq C_2(r^*, \phi^*) \frac{s^{p_m}}{|\phi|^{q_{m}}} 
\end{align*} 
where $C_2(r^*,\phi^*)$ is a constant satisfying the same property as $C_1(r^*, \phi^*)$ above.  Hence, we may choose $\phi^*>0$ large enough and then $r^*>0$ large enough so that for some positive constants $c_m, d_m$
\begin{align*}
L\psi^\pm_m(r, \theta) & \leq -c_m \frac{r^{p_m}}{|\phi|^{q_m}} +d_m .
\end{align*}
for all $(s, \phi)$ with $(r, \theta(r, \phi)) \in \mathcal{R}_m$.  Undoing the time change, we see that there exist positive constants $C_m, D_m$ such that on $\mathcal{S}_m$
\begin{align}
\label{L:psi_mb}
\mathcal{L} \psi_m^\pm(r, \theta) \leq -C_m \frac{r^{p_m+n}}{|\phi|^{q_m}} + D_m.  
\end{align}   

\subsubsection*{Region $\mathcal{S}_{j+3}$}
Here, we again use $\psi_{j+3}^\pm$ to also denote the function of
$(r, \theta)$ defined by $\psi_{j+3}^\pm(r, \phi)$.  The estimate in
this region is nearly identical to the one that precedes it, except that the lower bound in the definition of $\mathcal{S}_{j+3}$ is slightly different depending on the parity of $n$.  Nevertheless, we may essentially trace
through the inequalities in the previous case to see that the same estimates hold, except that the constants $C_1$ and $C_2$ in this case depend on, in addition to $\phi^*$ and $r^*$, $\eta^*$.  We may, however, still pick the parameters according to Remark \ref{rem:parameterchoice} to arrive at the desired estimate on $\mathcal{S}_{j+3}$
\begin{align}
\label{L:Lpsij3b}
\mathcal{L} \psi_{j+3}^\pm(r, \theta) \leq - C_{j+3} \frac{r^{p_{j+3}+n}}{|\phi|^{q_{j+3}}} + D_{j+3} 
\end{align}
for some positive constants $C_{j+3}, D_{j+3}$. 

\subsubsection*{Region $\mathcal{S}_{j+4}$}
The estimates in this case are also very similar to the previous ones.  In fact, fact they are a little easier since we include more terms of $L$ in the approximating operator $A$. In what follows, we again
use $\psi_{j+4}$ to denote the function of $(r, \theta)$ determined by
$\psi_{j+4}(r, \phi)$ where $(r, \phi)=(r, \phi_{j+3})$.  Notice
that for all $(r, \phi)$ with $(r, \theta(r, \phi) ) \in
\mathcal{S}_{j+4}$
\begin{align*}
  L \psi_{j+4}(r, \theta) &= L_{(r, \phi)} \psi_{j+4}(r, \phi)\\
  &= A \psi_{j+4}(r, \phi) + (L_{(r, \phi)}-A-N) \psi_{j+4}(r, \phi) + N \psi_{j+4}(r, \phi)\\
  & = - h_{j+4} r^{p_{j+4}} + (L_{(r, \phi)}-A-N) \psi_{j+4}(r, \phi)
  + N \psi_{j+4}(r, \phi)
\end{align*}   
where the operator $N$ was defined in \eqref{eqn:N2}.  Using the very same ideas in the previous two regions and recalling that
$\gamma^{(j+3)}_1 r^{-1} \partial_{\phi}$ is included in $A$ when
$n=2j+2$, we note that for all $(r, \phi)$ with $(r,
\theta(r, \phi)) \in \mathcal{S}_{j+4}$
\begin{align*}
 |(L_{(r, \phi)}-A-N) \psi_{j+4}(r, \phi)| + |N \psi_{j+4}(r, \phi)| \leq C( r^*, \eta^*, \phi^*) r^{p_{j+4}} 
\end{align*}  
where $C(r^*, \eta^*, \phi^* )$ is a constant which can be made arbitrarily small by picking $r^*, \eta^*, \phi^*$ according to Remark \ref{rem:parameterchoice}.  Thus choosing these parameters accordingly, we see that there exist constants $c_{j+4}, d_{j+4}>0$
 such that
\begin{align*}
L \psi_{j+4}(r, \theta) & \leq - c_{j+4} r^{p_{j+4}} + d_{j+4}
\end{align*}
for all $(s, \phi)$ with $(r(s, \phi), \theta(s, \phi)) \in
\mathcal{S}_{j+4}$.  Undoing the time change, we determine the
existence of positive constants $C_{j+4}, D_{j+4}$ such that on
$\mathcal{S}_{j+4}$
\begin{align}
\label{eqn:psij4b}
\mathcal{L} \psi_{j+4}(r, \theta) \leq -C_{j+4} r^{p_{j+4}+n} + D_{j+4}.
\end{align}

\subsection{Checking the specific Lyapunov bounds}
Here we show that we can pick parameters so that the conditions of
Proposition \ref{I-prop:localimglobal}  of Part I \cite{HerzogMattingly2013I} are satisfied when $\gamma \in (n, 2n)$ is
arbitrary.  By the estimates of the previous section, all we need is the following proposition.

\begin{proposition}\label{prop:existlu}
There exist positive constants $l_i$, $u_i$ such that
\begin{xalignat*}{2}
  &\psi_0(r, \theta) =r^p && (r, \theta) \in \mathcal{S}_0\\
  l_1 r^p \leq &\psi_1(r, \theta) \leq u_1 r^p \qquad\qquad &&(r, \theta) \in \mathcal{S}_1\\
  l_2 \frac{r^p}{|\theta|^{\frac{p}{n}}} \leq  &\psi_2(r, \theta) \leq u_2 \frac{r^p}{|\theta|^{q}} && (r, \theta) \in \mathcal{S}_2\\
  l_m \frac{r^{p_{1,m}}}{|\phi|^{q_{1,m}}}\leq &\psi_{m}(r, \phi) \leq u_m \frac{r^{p_m}}{|\phi|^{q_m}} &&(r, \theta(r, \phi)) \in  \mathcal{S}_m \\
  l_{j+4} r^{p_{1, j+4}} \leq& \psi_{j+4}(r, \phi) \leq u_{j+4}
  r^{p_{j+4}} && (r, \theta(r, \phi)) \in \mathcal{S}_{j+4}
\end{xalignat*}
where $m=3, \ldots, j+3$ and $(r, \phi)=(r, \phi_{j+3})$ in the last inequality.  
\end{proposition}

\begin{proof}[Proof of Proposition~\ref{prop:existlu}]
  The lower bounds have already been established and the upper bounds
  follow directly from the expressions derived for each $\psi_i$.
\end{proof}

 \section{Optimality}
\label{sec:optimality}

Recalling that $\mu$ denotes the invariant measure of
\eqref{eqn:polyZ} and $\rho$ its density with respect to Lebesgue measure on $\R^2$, in this section we prove Theorem
\ref{thm:limit_bound}.  Before giving the
precise details, let us give the intuitive idea behind the proof.  To study
the process $z_t$ defined by \eqref{eqn:polyZ} is a neighborhood of
the point at infinity, it is convenient to make a substitution which
maps the point at infinity to $0$ and $0$ to the point at infinity.
There are many changes of variables which accomplish precisely this,
but only one gives the desired bound on the invariant density:
$w_t=1/z_t^n$.  The reason for this choice is that the drift
of the process $w_t$ is non-zero and bounded at $w=0$.  In particular,
the invariant measure for the process $w_t$ cannot possibly vanish nor
can it blow up at $w=0$.  By construction, $|z|^{2n+2} \rho(z,
\bar{z})$ when written in the variables $(w, \bar{w})$ is this
invariant measure; hence, by positivity of this quantity as
$|z|\rightarrow \infty$, Theorem \ref{thm:limit_bound} would then
follow.
However in the proof of Theorem \ref{thm:limit_bound}, we will never
actually make the substitution $w_t=1/z_t^n$ described above because
the inversion of the mapping $w=1/z^n$ is multi-valued and this leads to
unnecessary complications.  Nonetheless, this transformation can be seen to motivate many of the manipulations performed.

In the proof of Theorem \ref{thm:limit_bound} we will need the following result which is a corollary of the proof of Theorem \ref{thm:lyapfunmadness}.  The result gives uniform bounds in the initial condition on return times to large compact sets of the process, time-changed to accommodate the ``substitution" $w_t=1/z_t^n$, determined by the adjoint $\mathscr{L}^*$.     
\begin{corollary}
\label{cor:Lstard}
Consider the stochastic differential equation on $\C \setminus \{ 0\}$
\begin{align*}
dz_t^* = -|z_t^*|^{-(n-1)}\bigg( \mathcal{P}(z_t^*, \overline{z_t^*}) +
\frac{\sigma^2(n+1)}{\overline{z^*_t}}\bigg) \, dt +
\sigma |z_t^*|^{-\frac{n-1}{2}}\, dB_t 
\end{align*}
where $\mathcal{P}(z, \bar{z})=z^{n+1}+ F(z, \bar{z})$, and $F$, $n$, $\sigma$ and $B_t$ are as in equation \eqref{eqn:polyZ}.  For
$\gamma >0$, let $S_\gamma = \inf\{t>0\, : \, |z_t^*|\leq \gamma \}$.
Then the stopped process $z_{t\wedge S_\gamma}^*$ is non-explosive and for each $\gamma >0$ sufficiently large we have:
\begin{align*}
\sup_{z\in \C\setminus \{0 \}} \PP_z [S_\gamma = \infty]=0.
\end{align*}  
Additionally, for each $t, \epsilon >0$ there exists $\gamma >0$ large enough so that 
\begin{align*}
\inf_{z\in \C\setminus \{0\}} \PP_z[S_\gamma \leq t ]\geq 1-\epsilon.  
\end{align*}
\end{corollary}

\begin{proof}[Proof of Corollary \ref{cor:Lstard}]
  Let $a\in \C$ be such that $a^n=-1$ and consider the process $a
  z_{t\wedge S_\gamma}^*$.  Our goal is to show that $az_{t \wedge S_\gamma}$ has a Lyapunov pair $(\Psi, \Psi^{1+\delta})$ for some $\delta >0$.  Non-explosivity will follow from Lemma \ref{I-l:existInvM} of Part I \cite{HerzogMattingly2013I}, and the remaining conclusions concerning the entrance times $S_\gamma$, $\gamma >0$, will follow from Proposition \ref{I-thm:entrance_time} of  \cite{HerzogMattingly2013I}.  Note first that the generator $M$ of $az_{t\wedge S_\gamma}^*$ is of the following form when written in polar coordinates $(r, \theta)$:
\begin{align*}
M = r L
\end{align*}   
where $L$ is of the form \eqref{eqn:genL}.  Hence, because of the form of $M$, our Lyapunov function $\Psi$ will be the same one constructed in Section \ref{sec:constructionPsigen}.  Upon replacing $n$ by $1$ in the inequalities \eqref{eqn:psi_0b1}, \eqref{eqn:psi_1b1II}, \eqref{eqn:psi_2b1}, \eqref{L:psi_mb}, \eqref{L:Lpsij3b}, and \eqref{eqn:psij4b}, and then applying Proposition
6.8, we see that the chosen $\Psi$ has the required local Lyapunov estimate for $(r, \theta) \in S_m$, $m=0,1, \ldots, j+4$: 
\begin{align*}
(M \Psi)(r, \theta) \leq - C\Psi(r,\theta)^{1+\delta} + D 
\end{align*}       
where $C,D$ and $\delta $ are positive constants.  Since the boundary flux contributions will have the appropriate sign, the result now follows.        
\end{proof}

\begin{proof}[Proof of Theorem \ref{thm:limit_bound}]
  First note that the generator $\mathscr{L}$ of the process $z_t$ has the following form when written in the variables $(z, \bar{z})$:
\begin{align*}
\mathscr{L} = \mathcal{P}(z, \bar{z}) \partial_z +
\overline{\mathcal{P}(z, \bar{z})} \partial_{\bar{z}} +
\sigma^2 \partial_z \partial_{\bar{z}} 
\end{align*}
where $\mathcal{P}(z, \bar{z})= z^{n+1} + F(z, \bar{z})$.  Let
$\mathscr{L}^*$ denote the formal adjoint of $\mathscr{L}$.  Motivated by the discussion of the substitution $w_t=1/z_t^n$ at the beginning of the section, define
$c(z, \bar{z}) = |z|^{2n+2} \rho(z, \bar{z})$ where $\rho$ is the invariant probability density function with respect to Lebesgue measure on $\R^2$.  Since $\mathscr{L}^*$ is elliptic and $\mathscr{L}^* \rho=0$, observe that $c$ is a
smooth function everywhere since $\rho$ is smooth everywhere.  To see which equation $c$ satisfies, write $\rho = c|z|^{-2n-2}$ and use the fact that $\mathscr{L}^*\rho=0$ to see that  
\begin{align*}
(|z|^{-2n-2} \mathcal{M}c)(z, \bar{z})=0 \, \text{ for } z\neq 0,
\end{align*}  
where $\mathcal{M}$ is of the form
\begin{align*}
  \mathcal{M} &= \mathscr{L}^*  - \frac{\sigma^2 (n+1)}{\bar{z}} \partial_{z} -
  \frac{\sigma^2 (n+1)}{z} \partial_{\bar{z}} +f(z, \bar{z})
\end{align*}
and the potential $f$ satisfies
\begin{align*}
 f(z, \bar{z})&= -  \partial_{z}(\mathcal{P}(z, \bar{z})) -  \partial_{\bar{z}}(\overline{\mathcal{P}(z, \bar{z})}) \\
 &\,\, +\frac{(n+1)}{z} \mathcal{P}(z, \bar{z}) +
  \frac{(n+1)}{\bar{z}} \overline{\mathcal{P}(z, \bar{z})}  +
  \frac{\sigma^2(n+1)^2}{|z|^2}.  
\end{align*}
In particular, we also note that $c$ solves the equation for $z\neq 0$
\begin{align}
\label{eqn:fc}
(|z|^{-n-1}\mathcal{M} c)(z, \bar{z}) =0.   
\end{align}
Using \eqref{eqn:fc}, we will now apply Feynman-Kac to obtain an expression
for $c(z, \bar{z})$ that can be analyzed as $|z|\rightarrow \infty$.

Now consider the time-changed process $z_{t\wedge S_\gamma}^*$, $\gamma >2$, introduced in Corollary \ref{cor:Lstard}.  Observe that the generator of
$z_{t\wedge S_\gamma}^*$ constitutes every term in $|z|^{-(n-1)}\mathcal{M}$ except for multiplication by the potential function $|z|^{-(n-1)} f(z, \bar{z})$ which is smooth in $(z,
\bar{z})$ for $z\neq 0$ and satisfies
\begin{align*}
  |z|^{-(n-1)} f(z, \bar{z}) = \mathcal{O}(1) \text{ as } |z|\rightarrow \infty.
\end{align*} 
Hence, in particular, $|z|^{-(n-1)}f(z, \bar{z})$ is bounded on the set $\{z\in
\C\,: \, |z| \geq \gamma\}$ for all $\gamma >0$.  Let $S_{\gamma, n}$
be the first exit time of $z_t^*$ from the annulus $A_{\gamma, n}=\{
\gamma <|z|< n \}.$ By Feynman-Kac, we have for $\gamma \geq 2$
\begin{align*}
c(z, \bar{z}) = \E_{(z, \bar{z})} c(z_{t \wedge S_{\gamma,
    n}}^*,\overline{z^*_{t \wedge S_{\gamma, n}}}) e^{\int_{0}^{t
    \wedge S_{\gamma, n}} |z_s^*|^{-(n-1)}f(z_s^*, \overline{z_s^*}) \, ds }, \, \, \,
z\in A_{\gamma, n}. 
\end{align*}  
By Corollary \ref{cor:Lstard}, we have that $z_{t\wedge S_\gamma}^*$
is non-explosive.  Thus by Fatou's lemma, taking the
$\liminf_{n\rightarrow \infty}$ of both sides of the above we obtain
for $|z|\geq \gamma \geq 2$ 
\begin{align*}
c(z, \bar{z}) &\geq \E_{(z, \bar{z})} c(z_{t \wedge
  S_{\gamma}}^*,\overline{z^*_{t \wedge S_{\gamma}}}) e^{\int_{0}^{t
    \wedge S_{\gamma}} |z_s^*|^{-(n-1)}f(z_s^*, \overline{z_s^*}) \, ds }.   
\end{align*}
Applying Corollary
\ref{cor:Lstard} again, for $\gamma >0$ large enough $S_\gamma <
\infty$ almost surely.  Hence, applying Fatou's lemma and taking the $\liminf_{t\rightarrow \infty}$ of both sides of the previous inequality we see that   
\begin{align}
\label{eqn:sub_ineq}
c(z, \bar{z}) &\geq \E_{(z, \bar{z})}
c(z_{S_{\gamma}}^*,\overline{z^*_{S_{\gamma}}})
e^{\int_{0}^{S_{\gamma}} |z_s^*|^{-(n-1)}f(z_s^*, \overline{z_s^*}) \, ds }. 
\end{align}
We now bound the right-hand side of \eqref{eqn:sub_ineq} from below.  Since $$(|z|^{-n-1}\mathcal{M}c)(z, \bar{z})=0 \text{ for } z\neq 0$$and the operator $|z|^{-n-1} \mathcal{M}$ is elliptic for $z\neq 0$, there exists a constant $C(\gamma) >0$ such that 
\begin{align*}
c(z, \bar{z}) \geq C(\gamma)>0, \qquad |z|=\gamma.  
\end{align*} 
Moreover, since $|z|^{-(n-1)}f$ is bounded for $|z|\geq \gamma$, there exists a constant $D(\gamma)>0$ such that 
\begin{align*}
||z|^{-(n-1)}f(z, \bar{z})|\leq D(\gamma), \qquad |z| \geq \gamma.
\end{align*}
Hence, we obtain
\begin{align*} 
c(z, \bar{z}) &\geq  C(\gamma) \E_{(z, \bar{z})} e^{-S_\gamma D(\gamma)}\\
&\geq  C(\gamma) \E_{(z, \bar{z})} e^{-S_\gamma D(\gamma)} 1_{\{S_\gamma \leq t \}}\\
&\geq C(\gamma) e^{-t D(\gamma)}\PP_{(z, \bar{z})} [S_\gamma \leq t] 
\end{align*}
where the inequality above holds for all $\gamma,t >0$.  Applying Corollary \ref{cor:Lstard} once more, we see that for each $t>0$ there exists $\gamma >0$ such that 
\begin{align*}
\inf_{|z| \geq \gamma} c(z, \bar{z}) >0
\end{align*} 
finishing the result.  
\end{proof}


\section{Generalized It\^{o}'s Formula}\label{sec:peskir}
\label{sec:gen_ito}
In this section, we give a differently packaged proof of a weaker version of
Peskir's extension of Tanaka's formula \cite{Peskir_07} which still
affords the structure needed to build the Lyapunov functions contained
in this paper and Part I \cite{HerzogMattingly2013I}.  Instead of making use of Tanaka's formula
as in \cite{Peskir_07}, we opt to mollify along interfaces where the
function is not $C^2$ and then take limits.  For convenience, we deal
solely with the case of a time-homogeneous diffusion process $\xi_t$
on $\R^m$ with generator $$\mathscr{L}=\sum_{j=1}^d f^j(\xi) \partial_{\xi^j} +\sum_{i,j=1}^d
\frac{1}{2}g^{ij}(\xi) \partial_{\xi^i \xi^j}^2$$ where $f^i, g^{ij}$ are locally
Lipschitz and the matrix $(g^{ij})$ is non-negative definite.  Furthermore,
assume that $\varphi \in C(\R^m:\R)$ is such that
\begin{align*}
\varphi(x) = \begin{cases}
\varphi_1(x) & \text{  } x^{m} \leq b(x^1, x^2, \ldots, x^{m-1})\\
\varphi_2(x) & \text{  } x^{m} \geq b(x^1, x^2, \ldots, x^{m-1})
\end{cases}
\end{align*}
where the $\varphi_i$'s are $C^2$ on the domains above and $b\in
C^2(\R^{m-1}: \R)$.  The case of finitely many non-intersecting
boundaries is a simple consequence of the following result.

\begin{theorem}\label{thm:genIto}
Let $\tau_n= \inf\{t>0\, : \, |\xi_t| \geq n \}$.  Then for all $\xi\in \R^m$, $n\in \N$ and all bounded stopping times $\upsilon$:
\begin{align}
  \E_\xi \varphi(\xi_{\upsilon \wedge \tau_n}) = \varphi(\xi) &+  \E_{\xi}
  \int_0^{\upsilon \wedge \tau_n}\big[ \tfrac{1}{2}\mathscr{L} \varphi(\xi_s^1,
  \ldots, (\xi_s^m)^+)+   \tfrac{1}{2}\mathscr{L} \varphi(\xi_s^1, \ldots, (\xi_s^m)^-) \big]\, ds+  \text{\emph{Flux}}(\xi, \upsilon, n)
\end{align}
where 
\begin{align*}
  &(\mathscr{L} \varphi)(\xi^1, \ldots, (\xi^m)^+) = \lim_{x^{m}\downarrow \xi^m} (\mathscr{L} \varphi)(\xi^1, \ldots,\xi^{m-1}, x^{m})\\
  & (\mathscr{L} \varphi)(\xi^1, \ldots, (\xi^m)^-) =
  \lim_{x^{m}\uparrow \xi^m} (\mathscr{L} \varphi)(\xi^1,
  \ldots,\xi^{m-1}, x^{m})
\end{align*}
and $\text{\emph{Flux}}(\xi,t,n)$ satisfies the following properties: 
\begin{itemize}
\item If $\partial_{x^{m}} \varphi_2(x)- \partial_{x^{m}} \varphi_1(x) \leq  0$ for all $x\in\R^m$ with $x^{m}=b(x^1, \ldots,x^{m-1})$, then $\text{\emph{Flux}}(\xi, \upsilon, n)\in (-\infty, 0]$ and $\text{\emph{Flux}}(\xi, t, k) \leq  \text{\emph{Flux}}(\xi, s, n)$ for $s\leq t$ and $n\leq k$. 
\item If $\partial_{x^{m}} \varphi_2(x)- \partial_{x^{m}} \varphi_1(x) \geq  0$ for all $x\in\R^m$ with $x^{m}=b(x^1, \ldots,x^{m-1})$, then $\text{\emph{Flux}}(\xi, \upsilon, n)\geq 0$ and the maps $\upsilon \mapsto \text{\emph{Flux}}(\xi, \upsilon, n)$, $n\mapsto  \text{\emph{Flux}}(\xi, \upsilon, n)$ are increasing.  
\end{itemize}
\end{theorem}   

\begin{proof}[Proof of Theorem~\ref{thm:genIto}]
  Because we will stop the process $\xi_t$ at time $\tau_n$, without
  loss of generality we may assume that $\varphi$ has compact support
  (e.g. in a ball centered at the origin with radius much larger than
  $n$).  Let $\chi:\R^m\rightarrow \R$ be a smooth mollifier, set
  $\chi_\epsilon(\xi)= \epsilon^{-m} \chi(\epsilon^{-1} \xi)$ and
  define
\begin{align*}
\varphi_\epsilon(\xi)  =\int_{\R^m} \chi_\epsilon(\xi-x) \varphi(x)\, dx.  
\end{align*} 
Applying Dynkin's formula we have
\begin{align}
\label{eqn:mol_dyn}
\E_\xi \varphi_\epsilon(\xi_{\upsilon\wedge \tau_n}) - \varphi_\epsilon(\xi)=
\E_{\xi}\int_0^{\upsilon\wedge \tau_n} (\mathscr{L} \varphi_\epsilon)(\xi_s) \,
ds.
\end{align}
To obtain the desired formula, we begin computing partial derivatives of $\varphi_\epsilon$.  To keep expressions compact, we will use $\partial_{\xi^j}$ to denote $\frac{\partial}{\partial \xi^j}$.    
Write
\begin{align*}
  \varphi_\epsilon(\xi) =\int_{U^-} \chi_\epsilon(\xi-x)
  \varphi_1(x)\, dx + \int_{U^+} \chi_\epsilon(\xi-x) \varphi_2(x)\,
  dx .
\end{align*}
where $U^{-} = \{x^{m} < b(x^1, \ldots, x^{m-1}) \}$ and $U^{+}= \{ x^{m} \geq b(x^1, \ldots, x^{m-1})\}$.  Integrate by parts once and use the fact that $\varphi_1$ and $\varphi_2$ agree on the boundary $\Gamma=\{x\in \R^m\,: \, x^{m} = b(x^1, \ldots, x^{m-1}) \}$ to see that 
\begin{align*}
\partial_{\xi^j} \varphi_\epsilon (\xi)  &=-\int_{U^-} \partial_{x^j} \chi_\epsilon (\xi-x)\varphi_1(x)\, dx - \int_{U^+}  \partial_{x^j}\chi_\epsilon (\xi-x) \varphi_2(x)\, dx \\
&= \int_{U^-} \chi_\epsilon(\xi-x) \partial_{x^j} \varphi_1(x)\, dy + \int_{U^+} \chi_\epsilon(\xi-x) \partial_{x^j} \varphi_2 (x)\, dx.  
\end{align*}
Using the equality on the previous line, apply $\partial_{\xi^i}$ to both sides and then integrate by parts in the same fashion to obtain 
\begin{align}\label{eqn:2ndder}\partial^2_{\xi^i \xi^j} \varphi_\epsilon(\xi)
&=\int_{U^-} \chi_\epsilon(\xi-x) \partial^2_{x^i x^j} \varphi_1(x)\, dx + \int_{U^+} \chi_\epsilon(\xi-x)
 \partial^2_{x^i x^j}  \varphi_2(x)\, dx
 \\\nonumber &\qquad + \int_\Gamma \chi_{\epsilon}(\xi-x)
\big(\partial_{x^j} \varphi_2 -\partial_{x^j}
  \varphi_1  \big)(x)\boldsymbol{\sigma}^i \, dS_\Gamma(x)
\end{align}
where $\Gamma = \{x: x^{m} = b(x^1, \ldots, x^{m-1}) \}$ and $\boldsymbol{\sigma}^i$ is
the $i$-th component of the unit surface normal vector $\boldsymbol{\sigma}=(-\nabla b(x),
1)/\sqrt{1+ |\nabla b(x)|^2}$ of $\Gamma$.  We now claim that for $x\in \Gamma$ and $j=1,\ldots, m-1$ 
\begin{align*}
(\partial_{x^j} \varphi_2-\partial_{x^j}
  \varphi_1)(x) = \big(\partial_{x^{m}}
  \varphi_2-\partial_{x^{m}} \varphi_1\big)(x) \boldsymbol{\sigma}^j \sqrt{1+|\nabla b(x)|^2}. 
\end{align*}  
To prove this claim, for $x\in \R^m$ define
\begin{align*}
h(x^1, \ldots, x^{m-1})= \varphi(x^1, \ldots, x^{m-1}, b(x^1, \ldots, x^{m-1})).
\end{align*}
Since $b\in C^1(\R^{m-1}: \R)$, $h\in C^1(\R^{m-1}: \R)$.  Moreover, $\varphi_i $ are $C^2$ on their closed domains of definition, each of which include the boundary $\Gamma$.  Hence, computing derivatives we see that for $j=1, \ldots, m-1$ and $i=1,2$ 
\begin{multline*}
\partial_{x^j} h(x^1, \ldots, x^{m-1})= (\partial_{x^j}\varphi_i)(x^1, \ldots, x^{m-1},b(x^1, \ldots, x^{m-1}))\\
 +(\partial_{x^m}\varphi_i)(x^1, \ldots, x^{m-1},b(x^1, \ldots, x^{m-1})) \partial_{x^j} b(x^1, \ldots, x^{m-1})  
\end{multline*}
for $i=1,2$.   Therefore 
\begin{align*}
0 &= \partial_{x^j} h (x^1, \ldots, x^{m-1}) -  \partial_{x^j} h (x^1, \ldots, x^{m-1})\\
&= (\partial_{x^j}\varphi_2)(x^1, \ldots, x^{m-1},b(x^1, \ldots, x^{m-1}))- (\partial_{x^j}\varphi_1)(x^1, \ldots, x^{m-1},b(x^1, \ldots, x^{m-1}))\\
 &\qquad+(\partial_{x^m}\varphi_2) (x^1, \ldots, x^{m-1},b(x^1, \ldots, x^{m-1}))\partial_{x^j}b(x^1, \ldots, x^{m-1})\\
 &\qquad-( \partial_{x^m}\varphi_1 )(x^1, \ldots, x^{m-1},b(x^1, \ldots, x^{m-1}))\partial_{x^j}b(x^1, \ldots, x^{m-1}),
\end{align*}  
from which the claim now follows.  Since $\boldsymbol{\sigma}^m=1/\sqrt{1+ |\nabla b|^2}$, the claim in particular allows us to write  
\begin{align}
  \partial^2_{\xi^i \xi^j} \varphi_\epsilon(\xi)
  &=\int_{U^-} \chi_\epsilon(\xi-x) \partial_{x^i x^j}^2 \varphi_1(x)\, dx + \int_{U^+}
  \chi_\epsilon(\xi-x)  \partial_{x^i x^j}^2 \varphi_2(x)\, dx\\ 
  \nonumber &\qquad + \int_\Gamma \chi_{\epsilon}(\xi-x)
  \big(\partial_{x^{m}} \varphi_2 -\partial_{x^{m}}
    \varphi_1 \big)(x)\boldsymbol{\sigma}^i\boldsymbol{\sigma}^j \sqrt{1+|\nabla b(x)|^2} \, dS_\Gamma(x).
\end{align}
for $i, j=1,2, \ldots, m$.

Let us now see what the computations above tell us.  Letting $*$ denote convolution,
we can now write \eqref{eqn:mol_dyn} as
\begin{align}
  \label{eqn:conv}&\E_\xi \varphi_\epsilon(\xi_{\upsilon\wedge \tau_n}) -
  \varphi_\epsilon(\xi) \\ 
  \nonumber &- \E_{\xi} \sum_{j=1}^m \int_0^{\upsilon\wedge \tau_n} f^j(\xi_s) (
  \chi_\epsilon * 1_{U^-} \partial_{\xi^j} \varphi_1 )(\xi_s) +
 f^j(\xi_s)  (\chi_\epsilon*1_{U^+} \partial_{\xi^j} \varphi_2)(\xi_s) \, ds
  \\ 
  \nonumber&- \frac{1}{2}\E_{\xi} \sum_{i,j=1}^m  \int_0^{\upsilon \wedge \tau_n} g^{ij}(\xi_s) (
  \chi_\epsilon * 1_{U^-} \partial_{\xi^i \xi^j} \varphi_1 )(\xi_s) +
 g^{ij}(\xi_s) (\chi_\epsilon*1_{U^+} \partial_{\xi^i \xi^j} \varphi_2)(\xi_s) \,
  ds\\ 
  \nonumber &=\frac{1}{2} \E_{\xi}\int_0^{\upsilon \wedge \tau_n}\int_\Gamma \big(\partial_{x^{m}}
    \varphi_2-\partial_{x^{m}} \varphi_1 \big)(x)\chi_\epsilon(\xi_s -x)\sqrt{1+|\nabla b(x)|^2}\sum_{i,j=1}^m
  g^{ij}(\xi_s)\boldsymbol{\sigma}^i\boldsymbol{\sigma}^j\, dS_\Gamma(x)ds
\end{align}
Since $f^i$, $g^{ij}$ are locally bounded, by dominated convergence we may pass the the limit as $\epsilon \downarrow 0$ through all integrals on the lefthand side to see that 
\begin{align*}
  &\E_\xi \varphi(\xi_{\upsilon \wedge \tau_n})-  \varphi(\xi) -\tfrac{1}{2}\E_{\xi}
  \int_0^{\upsilon \wedge \tau_n } [(\mathscr{L} \varphi)(\xi_s^1,
  \ldots, (\xi_s^{d})^+)]\, ds-\tfrac{1}{2}\E_{\xi} \int_0^{\upsilon \wedge \tau_n } [(\mathscr{L}
  \varphi)(\xi_s^1, \ldots, (\xi_s^{d})^-)]\, ds\\
  &=\lim_{\epsilon \downarrow 0 } \frac{1}{2}\E_{\xi}\int_0^{\upsilon\wedge \tau_n}\int_\Gamma \big(\partial_{x^{m}}
    \varphi_2-\partial_{x^{m}} \varphi_1 \big)(x)\chi_\epsilon(\xi_s -x)\sqrt{1+|\nabla b(x)|^2}\sum_{i,j=1}^m
  g^{ij}(\xi_s)\boldsymbol{\sigma}^i\boldsymbol{\sigma}^j\, dS_\Gamma(x)ds\\
  &:= \text{Flux}(\xi, \upsilon, n).  
  \end{align*} 
To see that
$\text{Flux}(\xi,\upsilon , n)$ has the claimed properties, note that since the matrix $(g^{ij})$ is
non-negative we have that
\begin{equation*}
   \chi_{\epsilon}(\xi_s-x)\sqrt{1+ |\nabla b(x)|^2} \sum_{i,j}
g^{ij}(\xi_s) \boldsymbol{\sigma}^i \boldsymbol{\sigma}^j\geq 0.  
\end{equation*}
Also, the surface measure $dS_\Gamma$ is a nonnegative measure.  In particular,  $\text{Flux}(\xi, \upsilon, n)$ satisfies all claimed properties of the result.    
\end{proof}

\section{Conclusions}
\label{sec:concluions}
The techniques developed in this and its accompanying work provide a
general framework for constructing a Lyapunov function well adapted to
the dynamics of a particular problem. This systematic approach began
in \cite{AKM_12}. Here, however, a significant number of advances have
been made, allowing us to both cover a much larger class of problems
and simplify many details in the analysis. In particular, the use of a
generalized Tanaka formula \cite{Peskir_07} greatly simplifies the
patching together of the piecewise-defined Lyapunov functions when
compared to the treatments of the similar situations in
\cite{AKM_12,BodDoe_12,GHW_11}.

A few considerations remain incomplete in this
work. Section~\ref{I-sec:spacing} in Part I
\cite{HerzogMattingly2013I} makes a compelling argument supported by
numerical simulations for the scaling of large excursions. It would be
interesting to add the missing details, producing a rigorous
argument. Related to this, it would be also interesting to see if one
could scale time and space so that in the limit of vanishing noise,
the system converges to a random distribution on the loops of the
underlying deterministic system. The limiting loop system would be in
the spirit of the random spider/graph processes considered in
\cite{FreidlinWentzell} and subsequent works.  A possible path one 
could take to achieve this (and also some of the results of this paper) is 
to work in the coordinate variable $w=1/z^n$ as done in Section~\ref{sec:optimality}.  One could then obtain path properties of the diffusion $w_t=1/z_t^n$ near the origin in the $w$-plane by controlling the martingale part using the exponential martingale inequality.  If present, the lower-order terms in the drift would then have to be dealt with, perhaps by using time-changes and Girsanov transformations and/or further substitutions inspired by those made in Section~\ref{sec:PiecewiseBeh} of this paper.

There also is a number of possible directions for generalization. Here we have only considered complex polynomials whose highest order term is the monomial $a z^{n+1}$. More generally, one could consider leading-order monomial terms of the form $a z^{k}\bar z^j$ where $k+j=n+1$. If
$k > j+1$, then the system with noise added can be proven to be stable by
essentially the same arguments used in this and its companion
paper \cite{HerzogMattingly2013I}.  In this case, the invariant measure will again have polynomial decay at
infinity.  If $k=j$, then the system is trivially stable if $a<0$ and trivially
unstable if $a>0$. In this case, the norm-squared $|z|^2$ is easily
shown to be a Lyapunov function if $a<0$. If $j<k$ then the
deterministic flow rotates towards the unstable directions and not
away from them as was the case when $k<j$.  Here one expects to be able
to prove that they system blows-up with probability one.  More interesting is the case when the leading order monomial is
replaced by a polynomial made up of terms which all have homogeneity
$n+1$ under the radial scaling $z \mapsto \lambda z$. This will
produce a richer collection of possibilities. Nonetheless, we expect
the ideas contained in these notes to be very useful in determining
and proving stability properties when noise is added.

Another possible direction of generalization would be to consider
state-dependent noise, e.g. $\sigma \, dB_t\mapsto \sigma(z,\bar{z}) \, dB_t$ where 
$\sigma(z, \bar{z})$ is a suitable polynomial.  In many cases, analysis of the resulting stochastic system should be possible so long as $\sigma(z, \bar{z})$ does not grow too fast at infinity (relative to the leading-order drift term $a z^{n+1}$).    
Also, the analysis may be greatly simplified in some cases by transforming to an equation  
with additive noise by a time change and/or substitution.  A more difficult
direction of generalization would be to consider higher dimensional
unstable ODEs under the addition of noise. Here the geometry of the
underlying, deterministic dynamics can be quite complicated, if not
chaotic. In this work, we relied on the simplicity of the underling
dynamics in our analysis. Understanding how different regions patch
together could be much more complicated, if not intractable, in higher dimensions.  The most interesting and wide-open direction to pursue
would be to consider an unstable deterministic PDE and show that it stabilizes under the addition of noise.

\section*{Acknowledgments}
We would also like to acknowledge partial support of the NSF through
grant DMS-08-54879 and the Duke University Dean's office.

\bibliographystyle{plain}
\bibliography{polyZ}

\end{document}